\documentclass[reqno,11pt]{amsart}
\usepackage{amsmath,amsfonts,amssymb,amsxtra,latexsym,amscd,enumerate,enumitem,amsthm,verbatim,here}
\usepackage{mathtools}
\usepackage{array}
\usepackage{booktabs}
\usepackage[table,dvipsnames]{xcolor}

\usepackage{cancel}

\usepackage[margin=1.2in]{geometry}
\numberwithin{equation}{section}

\usepackage{hyperref}
\definecolor{myrefred}{RGB}{180,0,0}
\definecolor{myciteblue}{RGB}{0,70,140} 
\hypersetup{
    colorlinks=true,
    pdfstartview=FitV,
    linkcolor=myrefred,      
    citecolor=myciteblue,    
    urlcolor=black
}
\definecolor{labelkey}{rgb}{0.6,0,0}

\newcommand{\R}{\mathbb{R}}

\newcommand{\C}{\mathbb{C}}

\newcommand{\la}{\langle}
\newcommand{\ra}{\rangle}
\newcommand{\dft}{{\widetilde{\mathcal{F}}} }
\newcommand{\Fr}{\mathrm{Fr} }
\newcommand{\refl}{{\mathcal{R}} }
\newcommand{\bigo}{{\mathcal{O}} }
\newcommand{\test}{\phi_{x_0,h} }
\newcommand{\testhi}{\psi_{\nu,x_0,h} }
\newcommand{\mapa}{S_{q,a} }
\newcommand{\mapb}{S_{\tilde q,b} }
\newcommand{\mapc}{S_{q,b} }

\numberwithin{equation}{section} 

\newtheorem{theorem}{Theorem}[section]
\newtheorem{lemma}[theorem]{Lemma}
\newtheorem{corollary}[theorem]{Corollary}

\newtheorem{proposition}[theorem]{Proposition}
\newtheorem{definition}[theorem]{Definition}
\theoremstyle{remark}
\newtheorem{remark}[theorem]{Remark}

\begin{document}
\title[Inverse Modified Scattering for Cubic NLS with a Repulsive Delta Potential]{Inverse Modified Scattering for Cubic NLS with a Repulsive Delta Potential}

\author{Mengyi Xie}
\address{Yale University}
\email{mengyi.xie@yale.edu}

\maketitle
\begin{abstract}
We study the one-dimensional cubic nonlinear Schr\"odinger equation with a repulsive delta potential and a localized inhomogeneous coefficient. We prove small-data modified scattering and construct the associated vector-valued modified scattering map, whose two components encode the coupling of the
frequencies $\xi$ and $-\xi$ induced by the point interaction. We show that this map determines both the strength of the delta potential and the inhomogeneous coefficient. Quantitatively, the delta strength is recovered with Lipschitz stability, while the inhomogeneous coefficients satisfy a H\"older stability estimate. To our knowledge, these are the first recovery and stability results for a modified scattering map in the presence of an external potential.
\end{abstract}

\setcounter{tocdepth}{1}
\tableofcontents

\section{Introduction}\label{sec:intro}
We study the one-dimensional cubic nonlinear Schr\"odinger equation with a repulsive delta potential and a localized inhomogeneous cubic coefficient, 
\begin{equation}\label{main}
    i\partial_t u\,=\,H_qu\,+\,\lambda\left(1+a\right)|u|^2 u,\qquad u(t=0)=u_0.
\end{equation}
Here $u:\R_t\times \R_x\to \C$, $\lambda\in \R$, $a:\R\to \R$ is localized and
\begin{equation}\label{eq:def_H}
    H_q=-\partial_x^2+q\delta,\quad q>0,
\end{equation}
is the standard self-adjoint realization of the repulsive delta Schr\"odinger operator. All objects associated with the linear Hamiltonian depend on the strength $q$. Whenever a single value of $q$ is fixed, we suppress this dependence and write, for example $H$, $U(t)$ and $\dft$ in place of $H_q$, $U_q(t)$ and $\dft_q$. We restore the $q$-subscripts only when comparing two different point interactions.

The equation \eqref{main} enjoys conservation of mass
\begin{equation}\label{eq:mass}
    \mathbf{M}(u)=\int_\R |u(t,x)|^2dx=\mathbf{M}(u_0),
\end{equation}
and energy
\begin{equation}\label{eq:energy}
    \mathbf{E}(u)=\int_\R |\partial_xu(t,x)|^2 dx+q|u(t,0)|^2 +\frac{\lambda}{2}\int (1+a)|u(t,x)|^4 dx =\mathbf{E}(u_0).
\end{equation}
The purpose of this paper is twofold. First, we establish small-data modified scattering for \eqref{main}. Second, we study the inverse problem for the corresponding modified scattering map, proving both uniqueness and quantitative stability in the recovery of the inhomogeneity $a$.

\medskip

We use the following class of coefficients. 
\begin{definition}\label{def:admissible}
    We say $a:\R\to \R$ is admissible if $a\in L^1\cap L^\infty$, $xa\in L^2$, and $a'\in L^1$ and define the admissible norm as
    \[
    \|a\|_{\mathrm{ad}}:=\|a\|_{L^1}+\|a\|_{L^\infty}+\|xa\|_{L^2}+\|a'\|_{L^1}. 
    \]
\end{definition}
Our first result is the small-data modified scattering theory.
\begin{theorem}\label{thm:modified_scattering}
    Fix $q>0$, $\lambda\in \R$, $\beta\in(0,\frac{1}{8})$, and let $a$ be admissible. There exists a positive $\varepsilon_0$, such that, if 
    \begin{equation}
        u_0\in \Sigma:= \left\{f\in H^1,\,\, xf\in L^2 \right\}, \qquad \|u_0\|_\Sigma:=\|u_0\|_{H^1}+\|xu_0\|_{L^2}\leq  \varepsilon_0,
    \end{equation}
    then the initial-value problem \eqref{main} admits a unique global solution $u\in  C([0,\infty);H^1(\R))$. Moreover,
    \begin{equation}\label{eq:decay_u_thm}
        \|u(t)\|_{L^\infty}\lesssim \varepsilon_0 \la t\ra ^{-\frac{1}{2}}, \quad t\geq 0.
    \end{equation}
    Finally, there exists a unique asymptotic profile $W_a=W_a(u_0)\in L^\infty$, such that
    \begin{equation}\label{eq:phase_correction}
        u(t)=(2it)^{-\frac{1}{2}} e^{\frac{ix^2}{4t}}W_a\left(\frac{x}{2t}\right) e^{-\frac{i\lambda}{2}|W_a\left(\frac{x}{2t}\right)|^2\log t}+\bigo_{L^\infty}(t^{-\frac{3}{4}+\beta}),\quad \text{as }t\to \infty. 
    \end{equation}
\end{theorem}
The model case $a=0, q=0$, is the classical one-dimensional cubic NLS. This equation is long-range: small localized solutions decay at the linear rate $t^{-\frac{1}{2}}$, but ordinary linear scattering fails and must be replaced by scattering with a logarithmic phase correction. This modified scattering theory has been obtained by several methods, including inverse scattering, normal forms, space-time resonances, and wave-packet testing; see, for instance, \cite{HN,DZ,LS,KP,IT}.

Modified scattering in the presence of a potential was later studied in several works. The work most related to our problem studies the homogeneous repulsive delta case $a=0, q>0$, by Masaki-Murphy-Segata \cite{MMS}. Their argument adapts the Hayashi–Naumkin \cite{HN} method to the distorted Fourier transform associated with the repulsive delta operator. In the normalization used in the present paper, $U=e^{-itH}$ is diagonalized by the distorted Fourier transform $\dft$ as
\begin{equation}
    U(t)=\dft^{-1} e^{-it\xi^2}\dft. 
\end{equation}
The free Schr\"odinger factorization is correspondingly written as 
\[
U_0(t)=M(t)D(t)\mathcal{F}M(t),
\]
where
\begin{equation}
    \left(M(t)f\right)(x)=e^{\frac{ix^2}{4t}}f(x),\qquad \left( D(t)f\right)(x)=(2it)^{-\frac{1}{2}}f\left(\frac{x}{2t}\right).
\end{equation}
For the delta flow, one defines $V(t)$ by
\begin{equation}\label{eq:def_V}
    U(t)=M(t)D(t)V(t)\dft.
\end{equation}
Equivalently, if 
\begin{equation}
    w(t)=\dft U(-t)u(t),
\end{equation}
then
\begin{equation}
    u(t)=M(t)D(t)V(t)w(t).
\end{equation}
This factorization is useful because it transfers the long-time analysis of $u$ to an equation for the distorted profile $w$. In the free case, $V_0(t)$ equals $e^{-\frac{i}{4t}\partial_x^2}$ and becomes $\mathrm{Id}$ as $t\to \infty$. Thus the leading profile equation is a scalar long-range ODE. For the delta potential, however, the asymptotics of $V(t)$ and $V^{-1}(t)$ involve both $w(\xi)$ and its reflected value $w(-\xi)$. Thus the leading dynamics is not a scalar equation, but a two-component system for
\begin{equation}\label{eq:def_vec_w}
    \vec w(t,\xi)=\big( w(t,\xi), \,\,w(t,-\xi)\big)^T.
\end{equation}
This system has a Hermitian leading matrix. Denoting the diagonalized and renormalized variables by 
\[
\vec g(t,\xi)=\big( g_1(t,\xi), \,\,g_2(t,\xi)\big)^T,
\]
then $\vec g$ converges as $t\to \infty$. The scalar profile $\phi_a(u_0)$ in Theorem \ref{thm:modified_scattering} is recovered from this limit vector after undoing the diagonalization and applying the asymptotics of $V(t)$.

\medskip

For the inverse problem, we retain the full two-component limiting profile rather than only the scalar physical-space profile. For each $q>0$ and admissible coefficient $a$, fix a small-data threshold $\varepsilon_{q,a}>0$ for which Theorem \ref{thm:modified_scattering} and the construction below hold, and define
\begin{equation}\label{eq:def_small_data_ball}
    \mathcal{B}_{q,a} :=\left\{u_0\in\Sigma:\|u_0\|_\Sigma\leq\varepsilon_{q,a}\right\}.
\end{equation}
We define the \emph{modified scattering map}
\begin{equation}\label{eq:def_scattering_map}
    \begin{split}
        S_{q,a}:\mathcal{B}_{q,a}&\longrightarrow
        \left(L^\infty\cap L^2\right)\times\left(L^\infty\cap L^2\right),\\
        u_0&\longmapsto \lim_{t\to\infty}\vec g(t)
        =\left(\lim_{t\to\infty}g_1(t),
        \lim_{t\to\infty}g_2(t)\right)^T.
    \end{split}
\end{equation}
The vector structure is essential: the delta potential creates reflected interactions, which are naturally encoded by the two components of $\vec w$. When $q$ is fixed, we abbreviate $S_{q,a}$ by $S_a$.

\medskip

Our second result states that this modified scattering map determines the strength of $\delta$ potential and the localized inhomogeneity. 
\begin{theorem}\label{thm:uniqueness}
    Fix $\lambda\neq 0$. Let $q$, $\tilde q>0$, and let $a$, $b$ be admissible coefficients. Let $S_{q,a}:\mathcal{B}_{q,a}\to \left(L^\infty\cap L^2 \right)\times\left(L^\infty\cap L^2 \right)$ and $S_{\tilde q, b}:\mathcal{B}_{\tilde q,b}\to \left(L^\infty\cap L^2 \right) \times \left(L^\infty\cap L^2 \right)$ be the corresponding modified scattering maps. If $S_{q,a}=S_{\tilde q,b}$ on $\mathcal{B}_{q,a}\cap \mathcal{B}_{\tilde q,b}$, then 
    \[
    q=\tilde q,\quad\text{ and }\quad  a=b \,\,\,\text{almost everywhere on }\R.
    \]
\end{theorem}

The recovery of nonlinearities and external potentials from scattering data has a long history, beginning with \cite{MS,EW,Weder97,Weder00,Weder01,Weder01b,Weder02,Watanabe01,SW,Sasaki07,Sasaki08,Watanabe18}. Further developments include analyticity-based approaches to nonlinear scattering maps, as in \cite{CG,PS}, as well as more recent recovery results for Schr\"odinger and wave equations in \cite{SUW, SS,KMV,Murphy,HMG,CKMV}. These works primarily concern settings in which the usual scattering operator or scattering map is available. 

The first recovery result in a genuinely modified-scattering regime is due to Chen–Murphy \cite{CM1}. They studied the one-dimensional cubic NLS with a localized inhomogeneous coefficient and no external potential, and proved that the modified scattering map uniquely determines the coefficient. In subsequent work, Chen–Murphy \cite{CM2} studied quantitative stability for this and several related nonlinear inverse problems with ordinary scattering. For the cubic modified-scattering problem, they obtained a \emph{logarithmic} stability estimate under an additional $L^2$ assumption on the derivatives of the coefficients.

\medskip

We next state a quantitative version of Theorem
\ref{thm:uniqueness}. The estimates arising from the proof naturally control the one-sided primitives of the inhomogeneous coefficients. For $x>0$, define
\begin{equation}\label{eq:def_A+}
    \mathcal{A}_+(x):=\int_x^\infty a(y)dy,\qquad \mathcal{B}_+(x):=\int_x^\infty b(y)dy,
\end{equation}
while for $x<0$, define
\begin{equation}\label{eq:def_A-}
    \mathcal{A}_-(x):=\int_{-\infty}^x a(y)dy,\qquad \mathcal{B}_-(x):=\int_{-\infty}^x b(y)dy.
\end{equation}
Set
\begin{equation}\label{eq:def_common_radius}
    \varepsilon_{q,\widetilde q,a,b}:=\min\left\{\varepsilon_{q,a},\varepsilon_{\widetilde q,b}\right\}.
\end{equation}
Thus both $S_{q,a}$ and $S_{\widetilde q,b}$ are defined whenever
$\|\varphi\|_\Sigma\leq\varepsilon_{q,\widetilde q,a,b}$. We measure the distance between the two maps by
\begin{equation}\label{def:SaSb}
    \big\|S_{q,a}-S_{\tilde q,b}\big\|:=\sup_{\substack{\varphi\in\mathcal{S}(\R)\\0<\|\varphi\|_\Sigma\leq \varepsilon_{q,\tilde q, a,b} }} \frac{\big\|S_{q,a}(\varphi)-S_{\tilde q,b}(\varphi)\big\|_{L^2\times L^2}}{\|\varphi\|_{\Sigma}}.
\end{equation}
The uniform $L^2$-bound for the limiting profiles shows that this quantity is finite.

\begin{theorem}\label{thm:stability}
Fix $\lambda\neq0$.  Let $q,\widetilde q>0$, and let $a,b$ be admissible. There is a constant $C_{q,\widetilde q}>0$ depending on $q$ and $\tilde q$, such that
\begin{equation}\label{eq:stability_q}
    |q-\widetilde q|\leq C_{q,\widetilde q}\,\|S_{q,a}-S_{\widetilde q,b}\|.
\end{equation}
Moreover, there are constants $\kappa_{q,\widetilde q,a,b}>0$ and $C_{q,\widetilde q,a,b}>0$ such that, whenever
$\|S_{q,a}-S_{\widetilde q,b}\|\leq\kappa_{q,\widetilde q,a,b}$,
\begin{equation}\label{eq:stability_A_B}
\begin{split}
    &\|\mathcal A_+-\mathcal B_+\|_{L^\infty(\mathbb R_+)}
      +\|\mathcal A_--\mathcal B_-\|_{L^\infty(\mathbb R_-)} \leq C_{q,\widetilde q,a,b}\,\|S_{q,a}-S_{\widetilde q,b}\|^{1/232}.
\end{split}
\end{equation}
The constants depend only on the displayed delta strengths, $\lambda$, and the
admissibility norms of $a$ and $b$.  If in addition $a',b'\in L^2(\mathbb R)$,
then
\begin{equation}\label{eq:stability_a_b}
    \|a-b\|_{L^\infty(\mathbb R)}
    \leq C_{q,\widetilde q,a,b}\,\|S_{q,a}-S_{\widetilde q,b}\|^{1/696},
\end{equation}
where the constant may also depend on $\|a'\|_{L^2}$ and $\|b'\|_{L^2}$.
\end{theorem}

The novelty of the present work is the treatment of an external
potential in a modified-scattering inverse problem. To the best of our knowledge, Theorems \ref{thm:uniqueness} and \ref{thm:stability} give the first uniqueness and quantitative stability results for a modified scattering map in the presence of an external potential. In contrast with the free problem, the repulsive delta potential changes the asymptotic representation of the linear flow: the usual Fourier
transform is replaced by the distorted Fourier transform associated with $H_q$, and transmission and reflection at the origin couple the modes $\xi$ and $-\xi$. Consequently, the natural modified scattering datum is a two-component limiting profile rather than a scalar profile. Our argument shows that this vector-valued datum determines both the strength of the delta potential and the localized inhomogeneous
coefficient. Moreover, it yields Lipschitz stability for the delta strength and a H\"older-type modulus for the inhomogeneous coefficient, rather than the logarithmic modulus obtained for the free modified-scattering problem in \cite{CM2}.


\subsection*{Notation}\label{subsec:notation}
Throughout the paper, all function spaces are over $\R$ unless otherwise indicated. We use the following conventions.

\begin{itemize}
    \item We write $\langle x\rangle:=\left(1+|x|^2\right)^{\frac12}$, $\R_+:=(0,\infty)$ and $\R_-:=(-\infty,0)$.
    For an interval $I\subseteq\R$, we denote by $C_c^\infty(I)$ the space of smooth functions compactly supported in $I$, and by $\mathcal{S}(\R)$ the Schwartz space.
    \item We use the standard mixed-norm notation
    \[
        \|F\|_{L_t^pL_x^q}:=\left( \int_{\R}\|F(t,\cdot)\|_{L_x^q}^p\,dt\right)^{\frac1p},
    \]
    with the usual modification when $p=\infty$.
    \item For nonnegative quantities $A$ and $B$, we write $A\lesssim B$ if $ A\leq CB$ for some constant $C>0$. Subscripts indicate the permitted dependence of the implicit constant; for example, $A\lesssim_{a}B$ means that the constant may depend on $\|a\|_{\mathrm{ad}}$. 
    \item If $X$ is a Banach space and $A\geq0$, then $F=\bigo_X(A)$ means that $ \|F\|_X\lesssim A.$ When $X=L^\infty$, we sometimes omit the subscript and write simply $\bigo(A)$. For scalar or pointwise error terms, subscripts on $\mathcal{O}$ indicate parameter dependence.
    \item We denote the reflection operator by $\refl f(x):=f(-x).$
    For a scalar function $f$, we write
    \[
        \vec f:=\bigl(f,\refl f\bigr)^T.
    \]
    For $\vec f=(f_1,f_2)^T\in X\times X$, we use the product norm $\|\vec f\|_{X\times X}:=\left(\|f_1\|_X^2+\|f_2\|_X^2\right)^{\frac12}.$
    \item Our complex $L^2$-pairing is
    \[
        \langle f,g\rangle_{L^2}
        :=
        \int_{\R}f(x)\overline{g(x)}\,dx,
    \]
    which is linear in the first argument. For vectors, we use the componentwise pairing
    \[
        \langle \vec f,\vec g\rangle_{L^2\times L^2}
        :=
        \langle f_1,g_1\rangle_{L^2}
        +
        \langle f_2,g_2\rangle_{L^2}.
    \]
    We distinguish this from the algebraic row-column pairing
    \[
        (a,b)\cdot(x,y)^T:=ax+by,
    \]
    which does not involve complex conjugation. When the entries are functions, this convention is applied pointwise.

    \item We denote the $2\times2$ diagonal matrix with entries $a,b$ by $\operatorname{diag}(a,b).$
    \item We denote the free Schr\"odinger group by $U_0(t)$, corresponding to $H_0=-\partial_{xx}$. 
\end{itemize}


\subsection*{Acknowledgment}
The author is grateful to Wilhelm Schlag for helpful discussions concerning the operator-theoretic aspects of Schr\"odinger operators with delta interactions. The author also thanks Gong Chen for a helpful conservation on the recovery of the delta strength.


\section{Linear theory for $H$: distorted Fourier transform and $V(t)$}\label{sec:lin}
\subsection{The delta Hamiltonian and the distorted Fourier transform} Throughout the paper we use the unitary Fourier transform
\begin{equation}\label{eq:flat_Fourier}
    \hat f(\xi)=\mathcal F f(\xi):=\frac{1}{\sqrt{2\pi}}\int_{\R}e^{-ix\xi}f(x)\,dx,
    \qquad
    f(x)=\frac{1}{\sqrt{2\pi}}\int_{\R}e^{ix\xi}\hat f(\xi)\,d\xi .
\end{equation}
Throughout Sections \ref{sec:lin} and \ref{sec:direct}, $q>0$ is fixed. In accordance with the convention mentioned in Section \ref{sec:intro}, we suppress the $q$-dependence of all objects. We denote by $ H=-\partial_x^2+q\delta$ the self-adjoint Schr\"odinger operator with a repulsive point interaction at the origin, see \cite{QM}. It is the operator associated with the closed nonnegative quadratic form
\begin{equation}\label{eq:delta_form}
    \mathfrak h[f]=\int_{\R}|f'(x)|^2\,dx+q|f(0)|^2,
    \qquad D(\mathfrak h)=H^1(\R).
\end{equation}
Equivalently,
\begin{equation}\label{eq:domain_H_delta}
    D(H)=\left\{f\in H^1(\R)\cap H^2(\R\setminus\{0\}):
    f'(0^+)-f'(0^-)=qf(0)\right\},
\end{equation}
and $Hf=-f''$ on $\R\setminus\{0\}$. Since $q>0$, $H$ has no negative eigenvalues and no discrete spectrum. Thus
\begin{equation}\label{eq:spectrum_H_delta}
    \sigma(H)=\sigma_{\mathrm{ac}}(H)=[0,\infty).
\end{equation}

\medskip

For $\xi\neq0$, the transmission and reflection coefficients are
\begin{equation}\label{eq:TR}
    T(\xi)=\frac{2i\xi}{2i\xi-q},
    \qquad
    R(\xi)=\frac{q}{2i\xi-q}.
\end{equation}
We use their continuous extensions at $\xi=0$. Thus $T(0)=0$ and $R(0)=-1$. The identities
\begin{equation}\label{eq:TR_identities}
    T=1+R,
    \qquad |T|^2+|R|^2=1,
    \qquad T\overline R+\overline T R=0,\qquad |T+R|=1
\end{equation}
will be used repeatedly. The generalized eigenfunction at energy $\xi^2$ is
\begin{equation}\label{eq:distorted_eigenfunction}
    e(x,\xi)=
    \begin{cases}
        T(|\xi|)e^{ix\xi},& x\xi\geq 0,\\
        e^{ix\xi}+R(|\xi|)e^{-ix\xi},& x\xi<0.
    \end{cases}
\end{equation}
It is continuous at the origin and satisfies
\begin{equation}\label{eq:eigenfunction}
    \partial_xe(0^+,\xi)-\partial_xe(0^-,\xi)=q e(0,\xi),
    \qquad He(\cdot,\xi)=\xi^2 e(\cdot,\xi)
\end{equation}
in the distributional sense. The distorted Fourier transform associated with $H$ is
\begin{equation}\label{eq:dft_definition}
    \widetilde \phi(\xi):=\dft\phi(\xi)
    :=\frac{1}{\sqrt{2\pi}}\int_{\R}\overline{e(x,\xi)}\phi(x)\,dx.
\end{equation}
It extends from $L^1\cap L^2$ to a unitary map on $L^2(\R)$, and its inverse is
\begin{equation}\label{eq:dft_inverse}
    \dft^{-1}\psi(x)=\frac{1}{\sqrt{2\pi}}\int_{\R}e(x,\xi)\psi(\xi)\,d\xi .
\end{equation}
Since $e(x,0)=0$, one has
\begin{equation}\label{eq:dft_zero}
    \dft f(0)=0
\end{equation}
whenever $f\in L^1(\R)$. The propagator $U(t)=e^{-itH}$ is diagonalized by $\dft$:
\begin{equation}\label{eq:linear_flow_dft}
    U(t)=\dft^{-1}e^{-it\xi^2}\dft.
\end{equation}
We recall the elementary comparison between $\dft$ and the flat Fourier transform.
\begin{lemma}\label{lem:relation}
For $\phi,\psi\in L^1(\R)$,
\begin{equation}\label{eq:dft_flat_relation}
    \dft\phi(\xi)=\mathcal F\phi(\xi)
    +\frac{1}{\sqrt{2\pi}}\overline{R(|\xi|)}
    \int_{\R}e^{-i|x||\xi|}\phi(x)\,dx,
\end{equation}
and
\begin{equation}\label{eq:dft_inverse_flat_relation}
    \dft^{-1}\psi(x)=\mathcal F^{-1}\psi(x)
    +\frac{1}{\sqrt{2\pi}}\int_{\R}e^{i|x||\xi|}R(|\xi|)\psi(\xi)\,d\xi .
\end{equation}
\end{lemma}
For proof, we refer to \cite{MMS}. The following estimates are the distorted analogues of the standard Fourier estimates.
\begin{lemma}\label{lem:dft_est}
For $t\in\R$,
\begin{equation}\label{eq:dft_est_1}
    \partial_\xi\dft U(t)=e^{-it\xi^2}(\partial_\xi-2it\xi)\dft .
\end{equation}
Moreover,
\begin{equation}\label{eq:dft_est_2}
    \|\partial_\xi\dft\phi\|_{L^2_\xi}\lesssim \|\langle x\rangle\phi\|_{L^2_x},
\end{equation}
\begin{equation}\label{eq:dft_est_3}
    \|\xi\dft\phi\|_{L^2_\xi}^2=\mathfrak h[\phi]
    =\|\partial_x\phi\|_{L^2_x}^2+q|\phi(0)|^2,
\end{equation}
\begin{equation}\label{eq:dft_L1_infty}
    \|\dft \phi\|_{L^\infty_\xi}\lesssim \|\phi\|_{L^1},\quad \|\dft^{-1}\psi\|_{L^\infty_x}\lesssim\|\psi\|_{L^1},
\end{equation}
and 
\begin{equation}\label{eq:dft_est_new}
    \|\partial_x\dft^{-1}\psi\|_{L^2_x}\le \|\xi\psi\|_{L^2_\xi}.
\end{equation}
Finally, if $\psi(0)=0$, then
\begin{equation}\label{eq:dft_inverse_weighted}
    \|x\dft^{-1}\psi\|_{L^2_x}\lesssim \|\psi\|_{H^1_\xi}.
\end{equation}
\end{lemma}
\begin{proof}
Identity \eqref{eq:dft_est_1} follows immediately from \eqref{eq:linear_flow_dft}. Estimates \eqref{eq:dft_est_2} and \eqref{eq:dft_inverse_weighted} are proved in \cite{MMS}. Estimate \eqref{eq:dft_L1_infty} follows directly from Lemma \ref{lem:relation} and the boundedness of the reflection coefficient in \eqref{eq:TR_identities}. The identity \eqref{eq:dft_est_3} is the spectral representation of the quadratic form \eqref{eq:delta_form}. Applying the same identity to $\phi=\dft^{-1}\psi$ gives
\[
    \|\partial_x\dft^{-1}\psi\|_{L^2_x}^2+q|\dft^{-1}\psi(0)|^2
    =\|\xi\psi\|_{L^2_\xi}^2,
\]
which implies \eqref{eq:dft_est_new}. 
\end{proof}
Chen-Pusateri \cite{CP2} proved the local smoothing estimate in the distorted Fourier transform set-up, compared to the classical Kato smoothing estimate. We record the version we need:
\begin{lemma}[Distorted local smoothing estimate]\label{lem:dft_local_smoothing} For $\mathfrak{m}(x,\xi):\R_x\times \R_\xi\to \C$ with 
\begin{equation}\label{eq:bounded_q}
    \sup_{x\in \R,\xi\in \R}|\mathfrak{m}(x,\xi)|<\infty,
\end{equation}
and $\phi:\R\to \C$ with $|\phi(\xi)|\lesssim \sqrt{|\xi|}$, we have
\begin{equation*}
    \left\|\int_1^t \int e^{is\xi^2} \phi(\xi)\, \overline{\mathfrak{m}(x,\xi)}F(s,x)dxds\right\|_{L^2_\xi}\lesssim \|F\|_{L^1_xL^2_s([1,t])}.
\end{equation*}
\end{lemma}
If $\mathfrak{m}(x,\xi)$ takes $e^{ix\xi}$, Lemma \ref{lem:dft_local_smoothing} is the usual Kato smoothing. \eqref{eq:TR_identities} implies that the distorted basis $e(x,\xi)$ satisfies \eqref{eq:bounded_q}, as the flat Fourier basis. 

We also record a few estimates for the free flow and flat Fourier transform. 
\begin{lemma}\label{lem:osc_W11}
Let $m\in W^{1,1}(\R)$. Then, for all $t\geq 0$,
\begin{equation}\label{eq:osc_flat_W11}
    \sup_{y\in\R}\left|\int_\R e^{-it\xi^2+iy\xi}m(\xi)\,d\xi\right|
    \lesssim \langle t\rangle^{-\frac{1}{2}}\|m\|_{W^{1,1}},
\end{equation}
 and
\begin{equation}\label{eq:osc_abs_W11}
    \sup_{y\in\R}\left|\int_\R e^{-it\xi^2+i|y||\xi|}m(\xi)\,d\xi\right|
    \lesssim \langle t\rangle^{-\frac{1}{2}}\|m\|_{W^{1,1}}.
\end{equation}
\end{lemma}
\begin{proof}
For $0\leq t\leq 1$, both estimates follow from $\|m\|_{L^1}$.  For $t\geq1$, \eqref{eq:osc_flat_W11} is the standard one-dimensional estimate for the Schr\"odinger phase, with amplitude of bounded variation.  To prove \eqref{eq:osc_abs_W11}, split the integral into $\xi>0$ and $\xi<0$ and apply \eqref{eq:osc_flat_W11} to each half-line. The possible jump at the origin is controlled by the $W^{1,1}$ norm.
\end{proof}

\begin{lemma}\label{lem:free_packet_localization}
Let $\varphi\in \mathcal S(\R)$. For every $N\geq0$, there exists $C_{N,\varphi}>0$ such that
\begin{equation}\label{eq:free_packet_pointwise}
    |U_0(t)\varphi(x)|\leq C_{N,\varphi}
    \begin{cases}
        \langle x\rangle^{-N},&0\leq t\leq1,\\[0.3em]
        t^{-\frac{1}{2}}\left\langle \dfrac{x}{2t}\right\rangle^{-N},&t\geq1.
    \end{cases}
\end{equation}
In particular, we have
\begin{equation}\label{eq:maximal_function_bdd}
    |U_0(t)\varphi(x)|\leq C_{\varphi}\, \la x\ra ^{-\frac{1}{2}},\quad \text{for }t\geq 0.
\end{equation}
Consequently,
\begin{equation}\label{eq:free_packet_uniform_integrability}
    \sup_{\nu\geq 10}\sup_{x\in\R}
    \int_0^\infty \left|U_0\left(\frac{s}{2\nu}\right)\varphi(x-s)\right|^4\,ds<\infty,
\end{equation}
 and the same estimate holds with $x-s$ replaced by $x+s$.
\end{lemma}
\begin{proof}
For $0\leq t\leq1$, the first bound follows from the identity
$U_0(t)\varphi=\mathcal F^{-1}(e^{-it\xi^2}\widehat\varphi)$ and repeated integration by parts in $\xi$.  For $t\geq1$, the kernel formula gives
\[
    U_0(t)\varphi(x)=\frac{e^{i\frac{x^2}{4t}}}{\sqrt{4\pi it}}
    \int_\R e^{-\frac{ixy}{2t}}e^{i\frac{y^2}{4t}}\varphi(y)\,dy,
\]
so integration by parts in $y$ gives the second bound in \eqref{eq:free_packet_pointwise}. \eqref{eq:maximal_function_bdd} follows from \eqref{eq:free_packet_pointwise} with $N=\frac{1}{2}$ for small time. Otherwise, we split into $1\leq t\leq |x|$ and $t>|x|$. The latter is clear from $t^{-\frac{1}{2}}$, and we handle the middle time case by
\[
|U_0(t)\varphi(x)|\lesssim t^{-\frac{1}{2}} \left(\frac{2t}{x}\right)^\frac{1}{2}\lesssim \la x\ra ^{-\frac{1}{2}}.
\]

We prove \eqref{eq:free_packet_uniform_integrability} now; the estimate with $x+s$ is identical.  On $s<2\nu$, use the first bound in \eqref{eq:free_packet_pointwise}:
\[
    \int_0^{2\nu}\langle x-s\rangle^{-4N}\,ds\lesssim 1,\quad \text{for }N \text{ sufficiently large.}
\]
On $s\geq2\nu$, use the second bound in \eqref{eq:free_packet_pointwise}
\[
    \int_{2\nu}^\infty \left|U_0\left(\frac{s}{2\nu}\right)\varphi(x-s)\right|^4ds
    \lesssim \int_{2\nu}^\infty \left(\frac{\nu}{s}\right)^2
    \left\langle \frac{\nu(x-s)}{s}\right\rangle^{-4N}ds. 
\]
For $x\leq \nu$, we have $|\frac{\nu(x-s)}{s}|$ is bounded below by $\frac{\nu}{2}$. Thus, we estimate the integral by
\[
\int_{2\nu}^\infty \left(\frac{\nu}{s}\right)^2
    \left\langle \frac{\nu(x-s)}{s}\right\rangle^{-4N}ds\lesssim \nu^{-4N} \int_{2\nu}^\infty \left(\frac{\nu}{s}\right)^2ds \lesssim 1,\quad \text{for }N>\frac{1}{4}.
\]
For $x>\nu$, change of variables $z=\nu(x-s)/s$ rewrites the integral as
\[
\frac{\nu}{x}\int_{-\nu}^{\frac{x}{2}-\nu} \la z\ra^{-4N}dz\lesssim 1,
\]
and proved \eqref{eq:free_packet_uniform_integrability}.
\end{proof}

We also have the analogue of Lemma \ref{lem:dft_local_smoothing} for the delta flow:
\begin{lemma}\label{lem:delta_flow}
For all $t\geq 0$, the delta flow $U(t)$ satisfies
\begin{equation}\label{eq:delta_flow_flat_decay}
\|U(t)\psi\|_{L^\infty_x}
\lesssim \la t\ra^{-\frac{1}{2}}\left(\|\psi\|_{L^1_x}+\|\psi\|_{\Sigma}  \right).
\end{equation}
\end{lemma}
\begin{proof}
    Using Lemma \ref{lem:relation}, we decompose
    \begin{align*}
        U(t)\psi(x)&=U_0(t)\psi(x)+\frac{1}{\sqrt{2\pi}}
        \mathcal F^{-1}\left(e^{-it\xi^2}\overline{R(|\xi|)}\int_\R e^{-i|y||\xi|}\psi(y)dy\right)\\
        &+\frac{1}{\sqrt{2\pi}}\int_\R e^{-it\xi^2+i|x||\xi|}R(|\xi|)\hat\psi(\xi)d\xi+\frac{1}{2\pi}\int_\R e^{-it\xi^2+i|x||\xi|} |R(|\xi|)|^2\int_\R e^{-i|y||\xi|}\psi(y)dyd\xi .
    \end{align*}
    The standard dispersive estimate and the $H^1$-unitarity of the free flow, gives
    \[
    \left\|U_0(t)\psi\right\|_{L^\infty_x}\lesssim \la t\ra^{-\frac{1}{2}}\left(\|\psi\|_{L^1}+\|\psi\|_{H^1}\right).
    \]
    Splitting over the two spatial half-lines and applying Plancherel yields 
    \[
    \left\|\int _\R e^{-i|y||\xi|}\psi(y)dy\right\|_{H^1_\xi}\lesssim \|\la x\ra \psi\|_{L^2}.
    \]
    Moreover, as functions of $\xi$, $R(|\xi|)$ and $|R(|\xi|)|^2$ and their derivatives are $L^2$ functions. Thus
    \begin{align*}
        \left\| \overline{R(\xi)}\int _\R e^{-i|y||\xi|}\psi(y)dy\right\|_{W^{1,1}_\xi}+\left\|R(|\xi|)\hat \psi(\xi)\right\|_{W^{1,1}_\xi}+\left\| |R(\xi)|^2\int _\R e^{-i|y||\xi|}\psi(y)dy\right\|_{W^{1,1}_\xi}\lesssim \|\la x\ra \psi\|_{L^2}.
    \end{align*}
    Applying Lemma \ref{lem:osc_W11}, we obtain \eqref{eq:delta_flow_flat_decay}.
\end{proof}


\subsection{The operators $V(t)$ and $V(t)^{-1}$} 
For $t>0$, define $M(t)$ and $D(t)$ by
\begin{equation}\label{eq:MD_definition_sec2}
    (M(t)f)(x)=e^{\frac{ix^2}{4t}}f(x),
    \qquad
    (D(t)f)(x)=(2it)^{-\frac{1}{2}}f\left(\frac{x}{2t}\right).
\end{equation}
The factorization
\begin{equation}\label{eq:def_V_sec2}
    U(t)=M(t)D(t)V(t)\dft
\end{equation}
defines $V(t)$. Equivalently,
\begin{equation}\label{eq:V_inverse_definition_corrected}
    V(t)^{-1}=\dft U(-t)M(t)D(t).
\end{equation}
Since $M(t)$, $D(t)$, $U(t)$, and $\dft$ are unitary on $L^2$, so are $V(t)$ and $V(t)^{-1}$. \cite{MMS} states
\begin{lemma}\label{lem:V_relation}
    We have the following identities
    \begin{equation}
        V(t)\psi(x)=\int \sqrt{\frac{it}{2\pi}} e^{-it(x^2+\xi^2)}e(tx,\xi) \psi(\xi)d\xi,
    \end{equation}
    and 
    \begin{equation}
        V(t)^{-1}\phi(\xi)=\int \overline{\sqrt{\frac{it}{2\pi}}e^{-it(x^2+\xi^2)}e(tx,\xi)}\phi(x)dx. 
    \end{equation}
\end{lemma}
Set
\begin{equation}\label{eq:S_vector_definition}
    \vec S(\xi):=(S_1(\xi),S_2(\xi)):=(T(|\xi|),R(|\xi|)),
    \qquad
    \overline{\vec S(\xi)}:=(\overline{S_1(\xi)},\overline{S_2(\xi)}).
\end{equation}
We also write $\vec f=(f,\refl f)^T$. We record the asymptotic behavior of $V(t)$ and $V(t)^{-1}$ from \cite{MMS}:
\begin{proposition}\label{prop:asy_V}
For $\phi,\psi\in H^1(\R)$ and $t>0$,
\begin{equation}\label{eq:asy_V}
    \left\|V(t)\psi-\vec S\cdot\vec\psi
    -2\Fr(\sqrt{2t}|\cdot|)\psi(0)\right\|_{L^\infty}
    \lesssim t^{-\frac{1}{4}}\|\psi\|_{H^1},
\end{equation}
and
\begin{equation}\label{eq:asy_V_inverse}
    \left\|V(t)^{-1}\phi-\overline{\vec S}\cdot\vec\phi
    -2\overline{\Fr(\sqrt{2t}|\cdot|)}\phi(0)\right\|_{L^\infty}
    \lesssim t^{-\frac{1}{4}}\|\phi\|_{H^1},
\end{equation}
where
\begin{equation}\label{eq:Fresnel_definition}
    \Fr(y)=\sqrt{\frac{i}{2\pi}}\int_{-\infty}^{-y}e^{-ix^2/2}\,dx .
\end{equation}
\end{proposition}
The Fresnel term satisfies
\begin{equation}\label{eq:Fresnel_decay}
    |\Fr(y)|\lesssim \langle y\rangle^{-1}\,\,\text{for }y>0,
    \qquad
    \Fr(0)=\frac12 .
\end{equation}
Indeed, for large $|y|$, this follows by writing
\[
    e^{-ix^2/2}=\frac{\partial_x(xe^{-ix^2/2})}{1-ix^2}
\]
and integrating by parts on $(-\infty,-y)$. 
\begin{corollary}\label{cor:estimate_V}
For $\psi\in H^1(\R)$ and $t>0$,
\begin{equation}\label{eq:V_zero_bound}
    |[V(t)\psi](0)|\lesssim t^{-\frac{1}{4}}\|\psi\|_{H^1},
\end{equation}
and
\begin{equation}\label{eq:V_Linfty_bound}
    \|V(t)\psi\|_{L^\infty}\lesssim \|\psi\|_{L^\infty}+t^{-\frac{1}{4}}\|\psi\|_{H^1}.
\end{equation}
\end{corollary}
\begin{proof}
The bound \eqref{eq:V_zero_bound} follows from \eqref{eq:asy_V}, $T(0)=0$, $R(0)=-1$, and $\Fr(0)=\frac{1}{2}$, since the two leading terms at the origin cancel. The $L^\infty$ bound follows from \eqref{eq:asy_V}, \eqref{eq:Fresnel_decay}, and the boundedness of $S_1,S_2$.\end{proof}

We also use the following Sobolev estimates from \cite{MMS}.
\begin{proposition}\label{prop:homo_sobolev_est}
For $\phi,\psi\in H^1(\R)$ with $\psi(0)=0$, and for $t>0$,
\begin{equation}\label{eq:V_H1_bound}
    \|V(t)\psi\|_{H^1}\lesssim \|\psi\|_{H^1},
\end{equation}
and
\begin{equation}\label{eq:V_inverse_H1_bound}
    \|V(t)^{-1}\phi\|_{H^1}\lesssim t^{\frac{1}{2}}|\phi(0)|+\|\phi\|_{H^1}.
\end{equation}
\end{proposition}

\section{The direct problem}\label{sec:direct}
In this section, we prove Theorem \ref{thm:modified_scattering}. The proof follows along \cite{MMS}, with some modifications to handle the inhomogeneous cubic term. 

\subsection{The leading homogeneous system} 
Set the distorted profile as 
\begin{equation}\label{eq:def_profile}
    w(t)=\dft U(-t)u(t),\quad \text{equivalently,}\quad u(t)=M(t)D(t)V(t)w(t).
\end{equation}
Then $w$ is a solution to 
\begin{equation}\label{eq:pde_w}
    i\partial_t w=\lambda \dft U(-t)\left(|u|^2u\right)+\lambda \dft U(-t)\left(a|u|^2u\right).
\end{equation}
Equivalently, 
\begin{equation}\label{eq:pde_w_rewrite}
    \begin{split}
        i\partial_t w&= \frac{\lambda}{2t} V(t)^{-1}\left(|V(t)w|^2 V(t)w\right) +\lambda \dft U(-t)\left(a|u|^2u\right).
    \end{split}
\end{equation}
By Proposition \ref{prop:asy_V} and evenness of $S_1$ and $S_2$, 
\begin{equation}\label{eq:V_approximation}
     \left\|Vw-\vec S\cdot \vec w\right\|_{L^\infty}\,+\,\left\|\refl Vw-\vec S\cdot \overrightarrow{\refl w}\right\|_{L^\infty} \,\lesssim \,t^{-\frac{1}{4}}\|w\|_{H^1}. 
\end{equation}
Using Corollary \ref{cor:estimate_V}, Proposition \ref{prop:homo_sobolev_est} together with
\[
    \big||z|^2z-|z'|^2z'\big|\lesssim |z-z'|(|z|^2+|z'|^2),
\]
we get
\begin{equation}\label{eq:homogeneous_expansion_first}
\begin{split}
    V(t)^{-1}(\big|V(t)w\big|^2V(t)w)
    &=\overline{S_1}\big|\vec S\cdot \vec w\big|^2\vec S\cdot \vec w+\overline{S_2}\big | \vec S\cdot \overrightarrow{\refl w}\big|^2\vec S\cdot \overrightarrow{\refl w}+\mathcal E_1(t),\\
    \refl V(t)^{-1}(\big|V(t)w\big|^2V(t)w)
    &=\overline{S_2}\big|\vec S\cdot \vec w\big|^2\vec S\cdot \vec w+\overline{S_1}\big | \vec S\cdot \overrightarrow{\refl w}\big|^2\vec S\cdot \overrightarrow{\refl w}+\mathcal E_2(t),
\end{split}
\end{equation}
where
\begin{equation}\label{eq:E12_bound}
    \|\mathcal E_1(t)\|_{L^\infty}+\|\mathcal E_2(t)\|_{L^\infty}
    \lesssim t^{-\frac{3}{4}}\|w\|_{H^1}^3+t^{-\frac{1}{4}}\|w\|_{H^1}\|w\|_{L^\infty}^2.
\end{equation}
Thus, for $\vec w=(w,\refl w)^T$,
\begin{equation}\label{eq:matrix-pde}
    i\partial_t \vec w
    =\frac{\lambda}{2t} A(\xi,\vec w)\vec w\,+\, \lambda \vec{\mathcal N}_a(t,\xi)+\mathcal{E}(t,\vec w),
\end{equation}
where 
\begin{equation}\label{eq:def_Na}
    \vec{\mathcal N}_a(t,\xi):=\bigl(\mathcal N_a(t,\xi),\mathcal N_a(t,-\xi)\bigr)^T, \qquad \mathcal N_a(t,\xi)=\dft U(-t)\big(a|u|^2 u \big)(\xi),
\end{equation}
and 
\begin{equation}\label{eq:error_E_est}
    \mathcal{E}(t,\vec w)=\bigo\big(t^{-\frac{7}{4}}\|w\|_{H^1} ^3+t^{-\frac{5}{4}}\|w\|_{H^1}\|w\|_{L^\infty}^2\big).
\end{equation}
The matrix $A$ has the rank-one representation 
\begin{equation*}
    \begin{split}
        A(\xi,\vec w)&=\big|\vec S\cdot \vec w\big|^2 \begin{bmatrix}
        \overline{S_1}\\[0.6em]
        \overline{S_2}
    \end{bmatrix}
    \begin{bmatrix}
        S_1,\,\,S_2
    \end{bmatrix}+
    \big|\vec S\cdot \overrightarrow{\refl w}\big|^2\begin{bmatrix}
        \overline{S_2}\\[0.6em]
        \overline{S_1}
    \end{bmatrix}
    \begin{bmatrix}
        S_2,\,\,S_1
    \end{bmatrix}.
    \end{split}
\end{equation*}
In particular $A$ is Hermitian. Moreover, by \eqref{eq:TR_identities}, 
\begin{equation}\label{eq:def_B}
    B=
    \begin{bmatrix}
        S_1 & \overline{S_2} \\[0.4em]
        -S_2 & \overline{S_1}
    \end{bmatrix}
\end{equation}
is unitary and 
\begin{equation}
    B^*AB= 
    \begin{bmatrix}
        \big|\vec S\cdot \vec w\big|^2 &0\\[0.4em]
        0& \big|\vec S\cdot \overrightarrow{\refl w} \big|^2
    \end{bmatrix}.
\end{equation}


\subsection{Global well-posedness.} We consider $u_0\in \Sigma$, $\|u_0\|_\Sigma\leq \varepsilon_0$, with $\varepsilon_0>0$ to be chosen sufficiently small, and let $u$ be the solution to \eqref{main} with initial data $u_0$. Standard local well-posedness in $H^1$, together with the usual continuation criterion, reduces global existence to an a priori $H^1$-bound. We therefore work on the maximal forward lifespan and establish estimates uniform up to its endpoint. Fix $\beta\in(0,\frac{1}{8})$ and impose the bootstrap assumption
\begin{equation}\label{eq:boot}
    \|w(t)\|_{L^\infty}+\la t\ra ^{-\beta} \|w(t)\|_{H^1}\leq 2K\varepsilon_0,\quad 0\leq t\leq T.
\end{equation}
We shall choose a constant $K\geq 1$, depending only on $q,\lambda,\beta$ and the admissible norms of $a$, such that the bootstrap holds for small time. Throughout the bootstrap argument, all implicit constants may depend on $q,\lambda,\beta$ and $a$, but not on $K$, $T$ and $\varepsilon_0$. We first record the immediate large-time consequence of the bootstrap. For $t\geq 1$, Corollary \ref{cor:estimate_V} gives
\begin{equation}\label{eq:u_decay_boot}
    \|u(t)\|_{L^\infty} = \big\|M(t)D(t)V(t)w(t)\big\|_{L^\infty}
    \lesssim t^{-\frac{1}{2}}\left(\|w(t)\|_{L^\infty}+t^{-\frac{1}{4}}\|w(t)\|_{H^1}\right)
    \lesssim K\varepsilon_0 t^{-\frac{1}{2}}.
\end{equation}
Thus the bootstrap implies the dispersive decay for $u$ in \eqref{eq:decay_u_thm} on large-time scale. Consequently, \eqref{eq:dft_L1_infty} gives
\begin{equation}\label{eq:Na_Linfty_bound}
    \|\vec{\mathcal N}_a(t)\|_{L^\infty_\xi}
    \lesssim \|a|u(t)|^3\|_{L^1_x}
    \lesssim_a K^3\varepsilon_0^3 t^{-\frac{3}{2}}.
\end{equation}
For small time, we use conservation laws \eqref{eq:mass} and \eqref{eq:energy}. With the 1D Gagliardo--Nirenberg inequality
\begin{align*}
    \big\|\partial_x u\big\|_{L^2}^2\leq \big| \mathbf{E}(u_0)\big|+C_a\|u(t)\|_{L^4}^4 \lesssim_a \varepsilon_0^2 + \|u(t)\|_{L^2}^3 \big\|\partial_x u\big\|_{L^2}.
\end{align*}
Therefore, with $\varepsilon_0$ sufficiently small, Sobolev embedding gives
\begin{equation}\label{eq:decay_u_small}
    \|u(t)\|_{L^\infty}\lesssim \|u(t)\|_{H^1}\lesssim_a \varepsilon_0,\quad t\geq 0.
\end{equation}
The standard virial computation and $x^2u\big|_{x=0}=0$ give
\begin{align*}
    \partial_t \big\|xu(t)\big\|_{L^2}^2 =4 \Im \int _\R x\overline{u(t,x)}\partial_xu(t,x)dx;
\end{align*}
hence
\begin{equation}\label{eq:xu_bound_small}
    \partial_t \big\|xu(t)\big\|_{L^2}\leq  \big\|\partial_xu(t)\big\|_{L^2}\lesssim_a \varepsilon_0,\qquad \text{and }\,\,\, \big\|xu(t)\big\|_{L^2}\lesssim_a \varepsilon_0, \quad \text{for }0\leq t\leq 1.
\end{equation}
We now transfer these small-time physical-space bounds to the distorted profile. Since
\[
w(t)=\dft U(-t)u=e^{it\xi^2}\dft u,
\]
unitarity of $\dft$ gives the conservation of mass for $w$. Using Lemma \ref{lem:dft_est}, for $0\leq t\leq 1$, 
\begin{equation}\label{eq:small_time_w}
    \begin{split}
        \|w(t)\|_{L^\infty}\lesssim\|w(t)\|_{H^1}&\lesssim \|w(t)\|_{L^2}+\|\partial_\xi w(t)\|_{L^2}\lesssim \varepsilon_0 +\big\|\partial_\xi \dft u(t)\big\|_{L^2}+2t \big\|\xi \dft u(t)\big\|_{L^2}\\
        &\lesssim \varepsilon_0 + \big\|\la x\ra u(t)\big\|_{L^2}+t \|u(t)\|_{H^1}\lesssim_a \varepsilon_0.
    \end{split}
\end{equation}
We let $K$ be the implicit constant in \eqref{eq:small_time_w} and \eqref{eq:decay_u_small}, \textit{i.e.}
\begin{equation}\label{eq:small_time_boot}
    \|w(t)\|_{L^\infty}+\la t\ra ^{-\beta} \|w(t)\|_{H^1}\leq K\varepsilon_0,\qquad \|u(t)\|_{H^1}\lesssim K\varepsilon_0,\quad \text{for }0\leq t\leq 1.
\end{equation}

\medskip

We next prove the large-time improvement on \eqref{eq:boot}. One has
\begin{lemma}\label{lem:close_boot}
    Under bootstrap assumption \eqref{eq:boot}, one has
    \begin{equation}\label{eq:improved_energy_boot}
        \|w(t)\|_{H^1}
    \le \|w(1)\|_{H^1}+C_aK^3\varepsilon_0^3\langle t\rangle^\beta,
    \end{equation}
    and 
    \begin{equation}\label{eq:improved_infty_boot}
        \|w(t)\|_{L^\infty}
    \le \|w(1)\|_{L^\infty}+C_aK^2\varepsilon_0^2.
    \end{equation}
\end{lemma}
Lemma \ref{lem:close_boot} closes the bootstrap \eqref{eq:boot} with small $\varepsilon_0$. Using \eqref{eq:u_decay_boot} and \eqref{eq:decay_u_small} we have
\begin{proposition}\label{prop:quant_smallness}
    Fix $\beta\in(0,\frac{1}{8})$. Let $\|u_0\|_\Sigma\leq \varepsilon_0$ and let $u$ denote the solution to \eqref{main} with initial data $u_0$. For sufficiently small $\varepsilon_0$, the profile $w$ as in \eqref{eq:def_profile} satisfies
    \begin{equation}
        \|w(t)\|_{L^\infty}+\la t\ra^{-\beta}\|w(t)\|_{H^1}\lesssim_a \varepsilon_0, \quad \text{for any }t\geq 0.
    \end{equation}
    Consequently, 
    \begin{equation}\label{eq:decay_u}
        \|u(t)\|_{L^\infty}\lesssim_a \la t\ra^{-\frac{1}{2}}\varepsilon_0, \quad \text{for any }t\geq 0.
    \end{equation}
\end{proposition}
The $H^1$-bound prevents finite-time breakdown, and the local theory therefore extends the solution globally.

\medskip

Before proving Lemma \ref{lem:close_boot}, we first need to prove an $H^1$-estimate for the localized forcing.
\begin{lemma}\label{lem:H1_forcing_a}
    Under bootstrap assumption \eqref{eq:boot}, one has
    \begin{equation}\label{eq:H1_forcing_a}
        \left\| \int_1^t \mathcal{N}_a(s)ds\right\|_{H^1_\xi}\lesssim K^3\varepsilon_0^3 t^\beta.
    \end{equation}
\end{lemma}
\begin{proof}
    Rewrite
    \[
    \int_1^t \dft U(-s)a|u(s)|^2 u(s)ds=\int_1^t e^{is\xi^2} \dft a|u(s)|^2 u(s)ds.
    \]
    The worst term arises when $\partial_\xi$ hits $e^{is\xi^2}$ since an extra $s$ is created. Otherwise, we bound by
    \begin{equation}
        \begin{split}
            \int_1^t \left\| \dft a|u(s)|^2 u(s)ds\right\|_{H^1_\xi}ds\lesssim \int_1^t \left\|\la x\ra a |u(s)|^2 u(s)\right\|_{L^2_x}ds\lesssim_a \int_1^t \|u(s)\|_{L^\infty}^3 ds\lesssim K^3 \varepsilon_0^3.
        \end{split}
    \end{equation}
    The bad term takes the form
    \begin{align*}
        \int_1^t e^{is\xi^2}2is\xi \dft a|u(s)|^2u(s)ds=\int_1^t e^{is\xi^2}2is\xi\, \chi(\xi) \dft a|u(s)|^2u(s)ds+\int_1^t e^{is\xi^2}2is\xi \left(1-\chi(\xi)\right) \dft a|u(s)|^2u(s)ds,
    \end{align*}
    where $\chi$ is a smooth cut-off with support $|\xi|\leq 1$. This breaks the integral into a low-frequency piece and a high-frequency piece. For low frequencies, since $\left|\xi\chi(\xi)\right|\leq \sqrt{|\xi|} $, Lemma \ref{lem:dft_local_smoothing} bounds
    \begin{equation}\label{eq:Na_low}
    \begin{split}
         \left\| \int_1^t e^{is\xi^2}2is\xi\, \chi(\xi) \dft a|u(s)|^2u(s)ds\right\|_{L^2_\xi}&\lesssim \big\|sa|u(s)|^2 u(s)\big\|_{L^1_xL^2_s([1,t])}\\
         \lesssim \|a\|_{L^1}&\left(\int_1^t s^2 \|u(s)\|_{L^\infty}^6 ds\right)^\frac{1}{2}\lesssim_a K^3\varepsilon_0^3 \log(t).
    \end{split}
    \end{equation}
    For the high-frequency piece, we first apply Lemma \ref{lem:relation} to it.
    \begin{align}
        \int_1^t e^{is\xi^2}2is\xi &\left(1-\chi(\xi)\right) \dft a|u(s)|^2u(s)ds= \int_1^t e^{is\xi^2}2s \left(1-\chi(\xi)\right) \mathcal{F}\partial_x\left( a|u(s)|^2u(s)\right)ds \label{eq:Na_high_flat} \\
        &\qquad +\frac{1}{\sqrt{2\pi }}\int_1^t\int e^{is\xi^2} 2is \xi \overline{R(|\xi|)} \left(1-\chi(\xi)\right) e^{-i|x||\xi|} a(x)|u(s,x)|^2 u(s,x)dxds. \label{eq:Na_high_tail}
    \end{align}
    We can apply Lemma \ref{lem:dft_local_smoothing} to \eqref{eq:Na_high_tail} and reduce it to the same estimate in \eqref{eq:Na_low} since
    \[
    \left|\xi R(|\xi|)\right|=\left|\frac{q\xi}{2i|\xi|-q}\right|\leq \frac{q}{2}.
    \]
    For \eqref{eq:Na_high_flat}, we first compute
    \begin{equation*}
        \partial_x\left( a|u(s)|^2u(s)\right)=a'|u(s)|^2u(s)+a\left(2|u(s)|^2 \partial_x u(s)+u(s)^2 \overline{\partial_x u(s)}\right).
    \end{equation*}
    Since $1-\chi(\xi)\leq \sqrt{|\xi|}$, Lemma \ref{lem:dft_local_smoothing} reduces the estimate to 
    \begin{align*}
        \left\| s\,\partial_x \left(a|u(s)|^2 u(s)\right)\right\|_{L^1_xL^2_s([1,t])}. 
    \end{align*}
    The first part $a'|u|^2u$ can be handled as \eqref{eq:Na_low} with $a'$ replacing $a$. With \eqref{eq:def_profile}, we compute
    \begin{align}\label{eq:partial_x_u}
        s\partial_x u(s,x)=\frac{ix}{2}u(s,x)+\frac{1}{2}\left[M(s)D(s)\partial_x\left(V(s)w(s)\right)\right](x).
    \end{align}
    Applying H\"older inequality to the second term bounds
    \begin{align*}
        \|a\|_{L^2} \left\| |u(s)|^2 \left|[M(s)D(s)\partial_x\left(V(s)w(s)\right)(x) \right|\right\|_{L^2_{s,x}([1,t]\times \R)}&\lesssim_a \|u(s)\|_{L^4_sL^\infty_x}^2\left\|\partial_x \left(V(s)w(s)\right)\right\|_{L^\infty_sL^2_x}\\
        &\lesssim _a K^3\varepsilon_0^3  \la t\ra ^\beta.
    \end{align*}
    Since $xa\in L^2$, the same argument can be applied to estimate the easy part $ixu$ by
    \begin{align*}
        \|xa\|_{L^2}\|u(s)^3\|_{L^2_{s,x}([1,t]\times \R)}\lesssim_a \|u(s)\|_{L^4_sL^\infty_x}^2\|u(s)\|_{L^\infty_s L^2_x}\lesssim _a K^2\varepsilon_0^3.
    \end{align*}
Summing up, we have \eqref{eq:H1_forcing_a}. 
\end{proof}
\begin{proof}[Proof of Lemma \ref{lem:close_boot}] 
    From \eqref{eq:pde_w_rewrite},
    \begin{equation*}
        w(t)=w(1)-\frac{i\lambda}{2}\int_1^t s^{-1}V(s)^{-1}\big|V(s)w\big|^2 V(s)wds-i\lambda\int_1^t \mathcal{N}_a(s)ds.
    \end{equation*}
    Using Lemma \ref{lem:dft_est}, Corollary \ref{cor:estimate_V}, Proposition \ref{prop:homo_sobolev_est} and $w(t,0)=0$, 
    \begin{equation*}
\begin{split}
    s^{-1}\left\|V(s)^{-1}(|V(s)w|^2V(s)w)\right\|_{H^1}
    &\lesssim s^{-1}\left(s^{\frac{1}{2}}|[V(s)w](0)|^3+\||V(s)w|^2V(s)w\|_{H^1}\right)\\
    &\lesssim s^{-\frac{5}{4}}\|w\|_{H^1}^3+s^{-1}\|w\|_{H^1}\|w\|_{L^\infty}^2.
\end{split}
\end{equation*}
Applying \eqref{eq:boot} and Lemma \ref{lem:H1_forcing_a}, since $0<\beta<\frac{1}{8}$, we proved \eqref{eq:improved_energy_boot}. For pointwise boundedness, we use \eqref{eq:matrix-pde}. By \eqref{eq:boot}, the tail is now bounded by $K^3\varepsilon_0^3 t^{-\frac{5}{4}+\beta}$. Since $A$ is Hermitian and $\lambda\in \R$, the first term does not contribute to the evolution of $|\vec w(t,\xi)|^2$:
\begin{equation*}
    \partial_t \big|\vec w(t,\xi)\big|^2 = 2\Re \, \left\la -i\lambda \vec {\mathcal{N}_a}(t,\xi), \vec w(t,\xi)\right\ra + \left\la \bigo(K^3\varepsilon_0^3 t^{-\frac{5}{4}+\beta}), \vec w(t,\xi)\right\ra. 
\end{equation*}
Applying \eqref{eq:Na_Linfty_bound} and \eqref{eq:boot}, we obtain
\begin{equation*}
    \big|\vec w(t,\xi)\big|^2-\big|\vec w(1,\xi)\big|^2\lesssim \int_1^t K^4\varepsilon_0^4 \left(s^{-\frac{3}{2}}+s^{-\frac{5}{4}+\beta}\right)ds\lesssim K^4\varepsilon_0^4. 
\end{equation*}
Taking supremum over $\xi$ and square root, we prove \eqref{eq:improved_infty_boot}.
\end{proof}


\subsection{Asymptotic behavior} \label{subsec:asy_phase}
We now show \eqref{eq:phase_correction}. Set $\vec f:=B^*\vec w$. Note \eqref{eq:TR_identities} implies
\begin{equation*}
    \big|\vec S\cdot \vec w\big|^2 =|f_1|^2,\qquad \big|\vec S\cdot \overrightarrow{\refl w}\big|^2 =|f_2|^2.
\end{equation*}
Hence \eqref{eq:matrix-pde} and boundedness of $S_j$ imply that $\vec f$ solves a diagonal system:
\begin{equation}\label{eq:pde_f}
    i\partial_t \vec f=\frac{\lambda}{2t}
    \begin{bmatrix}
        |f_1|^2& 0\\[0.4em]
        0& \,\,|f_2|^2
    \end{bmatrix}\vec f \,+\, \lambda B^*\vec {\mathcal{N}_a}(t)\,+\, B^*\mathcal E(t,\vec w).
\end{equation}
Define the normalized vector, $\vec g$, as
\begin{equation}\label{eq:def_g}
    g_j(t):=e^{i\Theta_j(t)}f_j(t),\quad \Theta_j(t) =\int_1^t \,\frac{\lambda}{2s}\,|f_j(s)|^2 ds,\quad j=1,2.
\end{equation}
Then $\vec g$ solves
\begin{equation}\label{eq:pde_g}
    i\partial_t \vec g=\lambda \mathrm{diag}(e^{i\Theta_1(t)}, \,\, e^{i\Theta_2(t)}) B^*\vec {\mathcal{N}_a}(t)\,+\, \bigo(\varepsilon_0^3 t^{-\frac{5}{4}+\beta}).
\end{equation}
Applying \eqref{eq:Na_Linfty_bound}, we have
\begin{equation}\label{eq:pde_g_est}
    i\partial_t \vec g=\bigo(\varepsilon_0^3 t^{-\frac{5}{4}+\beta}).
\end{equation}
Hence $\vec g(t)$ converges in $L^\infty$ at the rate of $t^{-\frac{1}{4}+\beta}$, as $t\to \infty$. We denote the limit as $\vec \phi$. Using 
\[
\left|e^{ix}-e^{iy}\right|\leq |x-y|,\quad \text{for }x,y\in \R,
\]
we have
\[
f_j(t)=G_j(t)+\bigo(\varepsilon_0^3t^{-\frac{1}{4}+\beta}),\qquad G_j(t):=e^{-\frac{i\lambda}{2}|\Phi_j(t)|^2 \log(t)}\Phi_j,
\]
where $\Phi_j$ equals $\phi_j$ up to multiplication by some phase factor. Returning to $\vec w=B\vec f$,
\begin{equation*}
    \vec w(t)=B [G_1,\,\,\,G_2]^T+\bigo(\varepsilon_0^3t^{-\frac{1}{4}+\beta}).
\end{equation*}
Using Proposition \ref{prop:asy_V} and Proposition \ref{prop:quant_smallness}, we have
\begin{equation*}
    V(t)w(t)=\vec S\cdot \vec w(t)+\bigo(\varepsilon_0t^{-\frac{1}{4}+\beta})\,=\,\left(S_1^2-S_2^2 \right)G_1(t)+\left(S_1\overline{S_2}+S_2\overline{S_1}\right) G_2(t)+\bigo(\varepsilon_0 t^{-\frac{1}{4}+\beta}).
\end{equation*}
Applying \eqref{eq:TR_identities}, we see that $S_1^2-S_2^2=S_1+S_2$ and $S_1\overline{S_2}+S_2\overline{S_1}=0$. Thus 
\begin{equation}
    V(t)w(t)(x)=e^{-\frac{i\lambda}{2}|W(x)|^2 \log(t)}W(x)+\bigo_{L^\infty}(t^{-\frac{1}{4}+\beta}), \quad t\to \infty, \quad W=\left(S_1+S_2\right)\Phi_1.
\end{equation}
Applying definition \eqref{eq:def_profile}, we have \eqref{eq:phase_correction}.


\section{The inverse problem}\label{sec:inverse}
In this section we prove Theorem \ref{thm:uniqueness}.  Throughout the section, $\lambda\neq0$ and $0<\beta<\frac{1}{12}$ are fixed. For $q>0$ and an admissible coefficient $a$, let $\mathcal B_{q,a}\subset\Sigma$ denote the set of initial data for which the solution of \eqref{main} is global and has the asymptotic behavior described in Theorem \ref{thm:modified_scattering}.  For $u_0\in\mathcal B_{q,a}$, we write $u_{q,a}$, $w_{q,a}$, $\vec f_{q,a}$, $\vec g_{q,a}$, and $\vec\phi_{q,a}$ for the corresponding solution, distorted profile, diagonalized profile, renormalized profile, and limiting profile. Thus
\begin{equation}\label{eq:def_inverse_scattering_map}
    S_{q,a}:\mathcal B_{q,a}\longrightarrow
    (L^\infty\cap L^2)\times(L^\infty\cap L^2),
    \qquad
    S_{q,a}(u_0):=\vec\phi_{q,a}.
\end{equation}
We first recover the strength of the point interaction from the linearization of $S_{q,a}$.  Once the two strengths are known to agree, we suppress the common $q$-dependence and recover the inhomogeneous coefficient by the cubic argument developed below. For later use, we introduce the diagonal phase matrix
\[
\mathcal{D}_{q,a}(t,\xi):=\mathrm{diag}(e^{i\Theta_1(t)}, \,\, e^{i\Theta_2(t)}),
\]
where $\Theta_{q,a,j}$ is defined as in \eqref{eq:def_g} for the coefficient $a$ and delta strength $q$. 


\subsection{Recovery of the point-interaction strength}
\label{subsec:recover_q} We first linearize the modified scattering map around the origin. 

\begin{lemma}\label{lem:linearization_scattering_map}
Let $q>0$, let $a$ be admissible, and let $\psi\in\mathcal S(\mathbb R)$.  Then
\begin{equation}\label{eq:linearization_scattering_map}
    S_{q,a}(\varepsilon\psi)=\varepsilon L_q\psi +\mathcal O_{L^\infty\times L^\infty} \bigl(C_{q,a} \varepsilon^3 \|\psi\|_\Sigma^3 \bigr)
\end{equation}
for all sufficiently small $\varepsilon>0$, where the linear operator $L_q$ is defined as
\begin{equation}\label{eq:def_Lq}
    L_q\psi:=B_q^*\overrightarrow{\dft_q\psi}.
\end{equation}
\end{lemma}
\begin{proof}
Since the phase vanishes at $t=1$, one has
\[
    S_{q,a}(\varepsilon\psi)=B_q^*\vec w_{q,a }(1)+ \int_1^ \infty \partial_t \vec g_{q,a}(t)\,dt.
\]
Duhamel's formula gives
\[
\begin{split}
    B_q^*\vec w_{q,a}(1) =\varepsilon L_q\psi -i\lambda B_q^*\int_0^1 \overrightarrow{\dft_qU_q(-t) \bigl[(1+a)|u_{q,a}|^2u_{q,a}\bigr]}\,dt.
\end{split}
\]
The $L^1\to L^\infty$ bound for $\dft_q$, conservation of mass, and the small-time $H^1$ bound imply that the integral is
$\mathcal O_{L^\infty\times L^\infty}(C_{q,a} \varepsilon^3 \|\psi\|_\Sigma^3)$. The same bound for the integral over $[1,\infty)$ follows from \eqref{eq:pde_g_est}.  This proves
\eqref{eq:linearization_scattering_map}.
\end{proof}
With this linearization lemma, we can recover $q$ from $S_{q,a}$:
\begin{proposition}\label{prop:recover_q}
Let $q,\widetilde q>0$ and let $a,b$ be admissible. If
$S_{q,a}=S_{\widetilde q,b}$ in $\mathcal{B}_{q,a}\cap\mathcal{B}_{\tilde q,b}$, then
$q=\widetilde q$.
\end{proposition}
\begin{proof}
Lemma \ref{lem:linearization_scattering_map} implies that
\begin{equation}\label{eq:Lq_equality}
    L_q\psi=L_{\widetilde q}\psi,
    \qquad \psi\in\mathcal S(\mathbb R).
\end{equation}
Choose a nonzero $\psi\in C_c^\infty((0,\infty))$.  For $\xi>0$, Lemma
\ref{lem:relation} gives
\[
    \dft_q\psi(\xi)=\overline{T_q(\xi)}\widehat\psi(\xi),
    \qquad
    \dft_q\psi(-\xi)
      =\widehat\psi(-\xi)+\overline{R_q(\xi)}\widehat\psi(\xi).
\]
Using \eqref{eq:def_B} and
$R_q\overline{T_q}+T_q\overline{R_q}=0$, we obtain
\begin{equation}\label{eq:Lq_second_component}
    [L_q\psi]_2(\xi)=T_q(\xi)\widehat\psi(-\xi),
    \qquad \xi>0.
\end{equation}
Since $\widehat\psi$ is analytic and not identically zero, it is nonzero almost
everywhere on $(-\infty,0)$.  Equations \eqref{eq:Lq_equality} and
\eqref{eq:Lq_second_component} therefore imply
$T_q=T_{\widetilde q}$ on $(0,\infty)$.  The explicit formula for $T_q$ then
gives $q=\widetilde q$.
\end{proof}


\subsection{Analysis of the modified scattering map} We extract the cubic part of the difference of two modified scattering maps in this subsection. For each test function $\psi\in\mathcal{S}(\R)$, we define the testing vector $\vec{h}_\psi=B^*\overrightarrow{ \dft \psi}$. This is the natural vector against which we pair the two modified scattering maps. 

\begin{lemma}\label{lem:born_cubic_replacement}
    Let $a$ be admissible and let $u_a$ be the solution to \eqref{main} with initial data $\varepsilon\psi$. Then
    \begin{equation}\label{eq:born_small_time_cubic}
        \int_0^1 \left|\left\la (1+a)\left(|u_a|^2 u_a -\varepsilon^3 \big|U(t)\psi\big|^2U(t)\psi\right),\,U(t)\psi\right\ra_{L^2_x}\right|dt\lesssim _{a,\psi}\varepsilon^5,
    \end{equation}
    and 
    \begin{equation}\label{eq:born_large_time_cubic}
        \int_1^\infty \left|\left\la a\left(|u_a|^2 u_a -\varepsilon^3 \big|U(t)\psi\big|^2U(t)\psi\right),\,U(t)\psi\right\ra_{L^2_x}\right|dt\lesssim _{a,\psi}\varepsilon^5.
    \end{equation}
\end{lemma}
\begin{proof}
    We denote $r_a(t):=u_a(t)-\varepsilon U(t)\psi $. Duhamel's formula and Proposition \ref{prop:quant_smallness} give
    \begin{align}
        \left\|r_a(t)\right\|_{L^2_x}\lesssim_a\int_0^t \|u_a(s)\|_{L^2}\|u_a(s)\|_{L^\infty}^2 ds\lesssim_{a,\psi} \varepsilon^3 \log(1+t).
    \end{align}
    Moreover, Lemma \ref{lem:delta_flow} implies
    \begin{align*}
        \|u_a(t)\|_{L^\infty}+\left\|\varepsilon U(t)\psi\right\|_{L^\infty}\lesssim_{a,\psi} \varepsilon\la t\ra ^{-\frac{1}{2}}.
    \end{align*}
    Hence, 
    \begin{align*}
        \left\||u_a|^2 u_a -\varepsilon^3 \big|U(t)\psi\big|^2U(t)\psi \right\|_{L^2}\lesssim \|r_a\|_{L^2}\left(\|u_a(t)\|_{L^\infty}^2+\left\|\varepsilon U(t)\psi\right\|_{L^\infty}^2\right)\lesssim \varepsilon^5 \la t\ra^{-1}\log(1+t).
    \end{align*}
    On the finite interval $0\leq t\leq 1$, we estimate 
    \begin{align*}
        \int_0^1 \left|\left\la (1+a)\left(|u_a|^2 u_a -\varepsilon^3 \big|U(t)\psi\big|^2U(t)\psi\right),\,U(t)\psi\right\ra_{L^2_x}\right|dt\lesssim \left(1+\|a\|_{L^\infty}\right) \|\psi\|_{L^2}\,\varepsilon^5.
    \end{align*}
    For $t\geq 1$, we instead use the localization of $a$. 
    \begin{align*}
         \int_1^\infty \left|\left\la a\left(|u_a|^2 u_a -\varepsilon^3 \big|U(t)\psi\big|^2U(t)\psi\right),\,U(t)\psi\right\ra_{L^2_x}\right|dt\lesssim_\psi \|a\|_{L^2} \,\varepsilon^5 \int_1^\infty t^{-\frac{3}{2}}\log(1+t)dt\lesssim_{a,\psi}\varepsilon^5.
    \end{align*}
\end{proof}

The following lemma isolates the extra difficulty created by the point interaction. Unlike the free problem in, \textit{e.g.}\cite{CM1}, where the profile equation is scalar and the nonresonant cubic error can be written as an explicit universal correction, the delta flow couples the two channels $\xi$ and $-\xi$. The diagonalization of the homogeneous long-range dynamics is therefore only asymptotic, leaving a remainder $\mathcal{E}(t,\vec w_a)$ with no simple explicit scalar representation. By \eqref{eq:error_E_est}, this remainder is still cubic in the profile, and hence appears at the same order as the localized term containing $a$. The key is to show $\mathcal{E}(t,\vec w_b)-\mathcal{E}(t,\vec w_a)$ is of higher order. 
\begin{lemma}\label{lem:cancel_homo_error}
    Let $u_a$ and $u_b$ be the two solutions with common initial data $\varepsilon\psi$. Then, 
    \begin{equation}\label{eq:propogate_difference_w}
        \|w_a(t)-w_b(t)\|_{L^\infty_\xi}\lesssim \varepsilon^3\la t\ra ^{\beta},\qquad \|w_a(t)-w_b(t)\|_{H^1_\xi}\lesssim \varepsilon^3\la t\ra ^{2\beta},\qquad \text{for }t\geq 0.
    \end{equation}
    Consequently,
    \begin{equation}
        \left|\int_1^\infty \left\la \mathcal{D}_a(t)B^*\mathcal{E}(t,\vec w_a)-\mathcal{D}_b(t)B^*\mathcal{E}(t,\vec w_b), \,\vec{h}_\psi \right\ra dt \right |\lesssim_{a,b,\psi}\varepsilon^5.
    \end{equation}
\end{lemma}
\begin{proof}
    We first prove the stability estimate for $w_a-w_b$. Duhamel’s formula gives
    \begin{align*}
        w_a(t)-w_b(t)=w_a(1)-w_b(1)-\,&\frac{i\lambda}{2} \int_1^t s^{-1}V(s)^{-1} \left(|V(s)w_a|^2 V(s)w_a-|V(s)w_b|^2 V(s)w_b \right)ds\\
        &-i\lambda \int_1^t \left(\mathcal{N}_a(s)-\mathcal{N}_b(s)\right)ds.
    \end{align*}
    The small-time bounds and Duhamel formula on finite interval $[0,1]$ gives
    \begin{align*}
        \|w_a(1)-w_b(1)\|_{H^1_\xi}+\|w_a(1)-w_b(1)\|_{L^\infty_\xi}\lesssim _{a,b,\psi}\varepsilon^3.
    \end{align*}
    Lemma \ref{lem:H1_forcing_a} applied separately to $a$ and $b$ gives
    \begin{align*}
        \left\| \int_1^t \left(\mathcal{N}_a(s)-\mathcal{N}_b(s)\right)ds\right\|_{H^1_\xi}\lesssim_ {a,b,\psi}\varepsilon^3 \la t\ra ^\beta.
    \end{align*}
    Similarly, \eqref{eq:Na_Linfty_bound} gives
    \begin{align*}
        \left\|\int_1^t \left(\mathcal{N}_a(s)-\mathcal{N}_b(s)\right)ds\right\|_{L^\infty_\xi}\lesssim_{a,b,\psi} \varepsilon^3.
    \end{align*}
    We first estimate the homogeneous difference in $L^\infty_\xi$. Using Proposition \ref{prop:homo_sobolev_est}, Corollary \ref{cor:estimate_V}, 
    \begin{align*}
        &\left\| s^{-1}V(s)^{-1} \left(|V(s)w_a|^2 V(s)w_a-|V(s)w_b|^2 V(s)w_b \right)\right\|_{L^\infty_\xi}\\
        &\quad \qquad \lesssim \varepsilon^2 s^{-1}\|w_a(s)-w_b(s)\|_{L^\infty}+\varepsilon^2 s^{-\frac{5}{4}+\beta}\|w_a(s)-w_b(s)\|_{L^\infty}+\varepsilon^2 s^{-\frac{5}{4}}\|w_a(s)-w_b(s)\|_{H^1}.
    \end{align*}
    Similarly, we estimate the difference in $H^1_\xi$:
    \begin{align*}
        &\left\| s^{-1}V(s)^{-1} \left(|V(s)w_a|^2 V(s)w_a-|V(s)w_b|^2 V(s)w_b \right)\right\|_{H^1_\xi}\\
        &\quad \qquad \lesssim \varepsilon^2s^{-1}\|w_a(s)-w_b(s)\|_{H^1}+\varepsilon^2s^{-1+\beta}\|w_a(s)-w_b(s)\|_{L^\infty}+\varepsilon^2s^{-\frac{5}{4}+2\beta}\|w_a(s)-w_b(s)\|_{H^1}.
    \end{align*}
    We estimate $\|w_a(t)-w_b(t)\|_{L^\infty_\xi}$ and $\|w_a(t)-w_b(t)\|_{H^1_\xi}$ by a weighted coupled Gr\"onwall argument. Indeed, one imposes
    \[
    M(T):=\sup_{1\leq t\leq T} t^{-\kappa} \|w_a(t)-w_b(t)\|_{L^\infty_\xi}+\sup_{1\leq t\leq T}t^{-\beta-\kappa}\|w_a(t)-w_b(t)\|_{H^1_\xi},\quad \kappa=K\varepsilon^2,
    \]
    and obtains
    \[
    M(T)\lesssim_\beta \varepsilon^3 +\left(\frac{\varepsilon^2}{\kappa}+\frac{\varepsilon^2}{\beta+\kappa}+\varepsilon^2\right)M(T).
    \]
    Choosing $K$ sufficiently large and $\varepsilon$ sufficiently small, separately, closes the bootstrap, provided
    \[
    3\beta+\kappa<\frac{1}{4}.
    \]
    Since $\beta \in (0,\frac{1}{12})$, this is a condition on the smallness of $\varepsilon$. Hence, we have
    \begin{align*}
        \|w_a(t)-w_b(t)\|_{L^\infty_\xi}\lesssim_{a,b,\beta,\psi}\varepsilon^3 t^{K\varepsilon^2}, \qquad \|w_a(t)-w_b(t)\|_{H^1_\xi}\lesssim_{a,b,\beta,\psi} \varepsilon^3 t^{\beta+K\varepsilon^2}.
    \end{align*}
    With $\varepsilon$ small enough, we proved \eqref{eq:propogate_difference_w}. 

    We now estimate the difference of homogeneous remainders. Write, for simplicity
    \[
    P_a(t):=\vec S\cdot \vec w_a(t),\quad T_a(t):=V(t)w_a(t)-P_a(t),\quad Q_a(t):=\vec S\cdot \overrightarrow{\refl w_a},\quad \widetilde{T_a}(t)=\refl V(t)w_a(t)-Q_a(t),
    \]
    and respectively for $b$. By \eqref{eq:V_approximation}, 
    \[
    \big\|T_a(t)\big\|_{L^\infty_\xi}+\big\|T_b(t)\big\|_{L^\infty_\xi}+\big\|\widetilde{T_a}(t)\big\|_{L^\infty_\xi}+\big\|\widetilde{T_b}(t)\big\|_{L^\infty_\xi}\lesssim_{a,b,\psi} \varepsilon t^{-\frac{1}{4}+\beta}.
    \]
    Moreover, 
    \[
    \big\|T_a(t)-T_b(t)\big\|_{L^\infty_\xi}+\big\|\widetilde{T_a}(t)-\widetilde{T_b}(t)\big\|_{L^\infty_\xi}\lesssim_{a,b,\beta,\psi} \varepsilon^3  t^{-\frac{1}{4}+2\beta}.
    \]
    Consider the first component of the homogeneous approximation error. It is enough to estimate 
    \begin{align*}
        &\quad\big\|\left(|P_a+T_a|^2 (P_a+T_a)-|P_a|^2 P_a\right)-\left(|P_b+T_b|^2 (P_b+T_b)-|P_b|^2 P_b\right)\big\|_{L^\infty_\xi}\\
        &\lesssim \left(\|P_a\|_{L^\infty}+\|P_b\|_{L^\infty}+\|T_a\|_{L^\infty}+\|T_b\|_{L^\infty}\right)^2 |T_a-T_b\|_{L^\infty}\\
        &\qquad\quad +\left(\|P_a\|_{L^\infty}+\|P_b\|_{L^\infty}+\|T_a\|_{L^\infty}+\|T_b\|_{L^\infty}\right)\left(\|T_a\|_{L^\infty}+\|T_b\|_{L^\infty} \right)\|P_a-P_b\|_{L^\infty}\\
        &\lesssim_{a,b,\psi}\,\varepsilon^5 t^{-\frac{1}{4}+2\beta}+\varepsilon^5 t^{-\frac{1}{4}+3\beta}.
    \end{align*} 
    The reflected component is identical. This estimate controls the contribution arising from the replacement of $V(t)w_{a}$ or $V(t)w_{b}$ by its leading expression. We still need to estimate the Fresnel term and the remainder in the asymptotic expansion of $V(t)^{-1}$. By Proposition \ref{prop:asy_V}, the two remaining contributions to $\mathcal{E}(t,\vec w_a)-\mathcal{E}(t,\vec w_b)$ are bounded by
    \begin{align*}
        t^{-\frac{3}{2}}\big|G_a(t,0)-G_b(t,0)\big|+t^{-\frac{5}{4}}\big\|G_a(t)-G_b(t)\big\|_{H^1}\lesssim _{a,b,\psi} t^{-\frac{7}{4}+4\beta}\varepsilon^5+t^{-\frac{5}{4}+2\beta}\varepsilon^5,
    \end{align*}
    where 
    \[
    G_a(t):=|V(t)w_a(t)|^2 V(t)w_a(t),\quad G_b(t):=|V(t)w_b(t)|^2 V(t)w_b(t),
    \]
    for simplicity. Therefore, we conclude
    \begin{equation*}
        \left\|\mathcal{E}(t,\vec w_a)-\mathcal{E}(t,\vec w_b)\right\|_{L^\infty_\xi}\lesssim_{a,b,\psi} \varepsilon^5 t^{-\frac{5}{4}+3\beta},\quad t\geq 1.
    \end{equation*}
    Since $\beta <\frac{1}{12}$, this is integrable in time. Also $\vec h_\psi\in L^1_\xi\times L^1_\xi$, as $\psi\in\mathcal{S}$. Thus, 
    \begin{align*}
        \left|\int_1^\infty \left\la B^*\mathcal{E}(t,\vec w_a)-B^*\mathcal{E}(t,\vec w_b), \,\vec{h}_\psi \right\ra dt \right |\lesssim_{a,b,\psi}\varepsilon^5.
    \end{align*}
    It remains only to account for the phase matrices. From the direct estimates,
    \begin{align*}
        \left\| \mathcal{D}_a(t)-I\,\right\|+\left\| \mathcal{D}_b(t)-I\,\right\|\lesssim \sum_{j=1}^2 \left\|\Theta_{a,j}(t)\right\|_{L^\infty}+\left\|\Theta_{b,j}(t)\right\|_{L^\infty}\lesssim_{a,b,\psi}\varepsilon^2 \log(t),\quad t\geq 1,
    \end{align*}
    with 
    \begin{align*}
        \left\|\mathcal{E}(t,\vec w_a)\right\|_{L^\infty_\xi}+ \left\|\mathcal{E}(t,\vec w_b)\right\|_{L^\infty_\xi}\lesssim_{a,b,\psi} \varepsilon^3 t^{-\frac{5}{4}+\beta}.
    \end{align*}
    Hence, the contribution with $\mathcal{D}_a(t)-I$ or $\mathcal{D}_b(t)-I$ is $\bigo(\varepsilon^5)$. This proves the lemma.
\end{proof}

We are now ready to identify the third-order action of the modified scattering map.
\begin{proposition}\label{prop:difference_third_order}
    Let $a,b$ be admissible and let $\psi\in\mathcal{S}(\R)$. Then, for sufficiently small $\varepsilon>0$, 
    \begin{equation}\label{eq:difference_third_order}
        \left \la \,S_a(\varepsilon\psi)-S_b(\varepsilon\psi) ,\,\vec{h}_\psi\right\ra=-2i\lambda\varepsilon^3 \int_0^\infty \int (a-b)\big| U(t)\psi\big|^4 dxdt+\bigo_{a,b,\psi}(\varepsilon^5).
    \end{equation}
\end{proposition}
\begin{proof}
    We write 
    \[
    S_a(\varepsilon\psi)-S_b(\varepsilon\psi)=\vec g_a(1)-\vec g_b(1)+\int_1^\infty \partial_t \left(\vec g_a(t)-\vec g_b(t)\right)dt.
    \]
    We first consider the contribution at $t=1$. Since the phase functions $\Theta_{a,j}(1)=0=\Theta_{b,j}(1)$,
    \begin{align*}
        \left \la \, \vec g_a(1)-\vec g_b(1), \vec h_\psi \,\right\ra&=-2i\lambda \int_0^1 \left \la \, (1+a)|u_a(t)|^2 u_a(t)-(1+b)|u_b(t)|^2 u_b(t) , U(t)\psi \right \ra dt.
    \end{align*}
    By Lemma \ref{lem:born_cubic_replacement}, we further write this as
    \begin{align*}
        \left \la \, \vec g_a(1)-\vec g_b(1), \vec h_\psi \,\right\ra&=-2i\lambda\varepsilon^3 \int_0^1\int  (a-b)\big|U(t)\psi\big|^4 dxdt +\bigo_{a,b,\psi}(\varepsilon
        ^5).
    \end{align*}
    Here the homogeneous 1-part cancels after subtracting the two coefficients. 

    We next consider the integral $(1,\infty)$. Using uncompressed equation for $\vec g_a$:
    \[
    i\partial_t \vec g_a=\lambda \mathcal{D}_a B^*\vec {\mathcal{N}}_a+\mathcal{D}_a B^*\mathcal{E}(t,\vec w_a),
    \]
    and the respective equation for $\vec g_b$, we obtain
    \begin{align*}
        \int_1^\infty \partial_t \left(\vec g_a-\vec g_b\right)dt=-i\lambda \int_1^\infty &\left(\mathcal{D}_a(t) B^*\vec{\mathcal{N}_a}(t)-\mathcal{D}_b(t) B^*\vec{\mathcal{N}_b}(t)\right)dt\\
        &\qquad -i\int_1^\infty \left(\mathcal{D}_a B^*\mathcal{E}(t,\vec w_a)- \mathcal{D}_b B^*\mathcal{E}(t,\vec w_b)\right)dt.
    \end{align*}
    The second integral contributes $\bigo(\varepsilon^5)$, after pairing with $\vec h_\psi$, by Lemma \ref{lem:cancel_homo_error}. For the first term, we first remove the phase matrices as in the proof of Lemma \ref{lem:cancel_homo_error}, since 
    \[
    \left\|\mathcal{N}_a(t)\right\|_{L^\infty}\lesssim_{a,\psi}\varepsilon^3 t^{-\frac{3}{2}},\qquad \left\|\mathcal{N}_b(t)\right\|_{L^\infty}\lesssim_{b,\psi}\varepsilon^3 t^{-\frac{3}{2}}.
    \]
    Thus, using first \eqref{eq:def_Na} and then Lemma \ref{lem:born_cubic_replacement}, we have
    \begin{align*}
        -i\lambda \int_1^\infty &\left \la \,\mathcal{D}_a(t) B^*\vec{\mathcal{N}_a}(t)-\mathcal{D}_b(t) B^*\vec{\mathcal{N}_b}(t), \, \vec h_\psi\right\ra dt \\
        &\qquad \qquad = -2i\lambda\int_1^\infty \left \la\, a|u_a|^2u_a -b|u_b|^2 u_b ,\,U(t)\psi\right\ra dt+\bigo_{a,b,\psi}(\varepsilon^5)\\
        &\qquad \qquad =-2i\lambda \varepsilon^3\int_1^\infty \int\, (a-b)\big|U(t)\psi\big|^4 dxdt + \bigo_{a,b,\psi}(\varepsilon^5).
    \end{align*}
    Combining the small-time and large-time contributions gives \eqref{eq:difference_third_order}.
\end{proof}


\subsection{High-velocity wave packet testing} 
The final step is to extract $a-b$ from the quartic identity. In the free problem treated by Chen–Murphy \cite{CM1}, this is essentially algebraic: testing by the Gaussian function $\mathfrak{g}$ and its translation with the explicit expression of $U_0(t)\mathfrak{g}$ gives the uniqueness of $a$ at every point. The delta potential makes this step substantially less transparent. The distorted Fourier representation couples $\xi$ and $-\xi$, and a direct expansion into flat Fourier pieces produces reflected oscillatory tails with phases involving $|x|$ and $|\xi|$, rather than a clean ray transform of $a-b$.

Our resolution is to test with the high-velocity wave packets supported on one side of the origin and moving away from the delta interaction. With this, we prove
\begin{proposition}\label{prop:recover_a_delta_flow}
    Let $a\in L^1(\R)$. Suppose that
    \begin{equation}\label{eq:assumption_unique}
        \int_0^\infty a\big|U(t)\psi\big|^4 dxdt=0,\qquad \text{for }\psi\in \mathcal{S}(\R),
    \end{equation}
    then $a=0$ almost everywhere on $\R$.
\end{proposition}
\begin{proof}
    We first recover $a$ on the positive half line. Let $\phi\in C_c^\infty((0,\infty))$ and set
    \[
    \psi_\nu(x)=e^{ix\nu}\phi(x),\quad \nu\geq 10.
    \]
    The packet $\psi_\nu$ is supported on the right and has positive velocity, so under the free flow it moves further to the right. 

    We begin by comparing the delta flow with the free flow. Lemma \ref{lem:relation} gives
    \begin{align*}
        \dft \psi_\nu=\widehat{\psi_\nu}+\overline{R(|\xi|)} \hat \phi(|\xi|-\nu).
    \end{align*}
    Moreover, we have $U(t)\psi_\nu$ equals $U_0(t)\psi_\nu$ plus a finite sum of oscillatory remainder terms of the form
    \begin{align*}
        \int e^{-it\xi^2}e^{ix\xi}m_\nu(\xi)d\xi,\quad \text{or}\quad \int e^{-it\xi^2}e^{i|x||\xi|}m_\nu(\xi)d\xi,
    \end{align*}
    with $m_\nu$ being one of
    \[
    \overline{R(|\xi|)} \hat \phi(|\xi|-\nu),\quad |R(|\xi|)|^2 \hat \phi(|\xi|-\nu), \quad R(|\xi|)\hat \phi(\xi-\nu).
    \]
    One can check easily that all of these functions are $W^{1,1}_\xi$ with
    \[
    \left\|m_\nu\right\|_{W^{1,1}_\xi}\lesssim_\phi \nu^{-1},
    \]
    from $R(|\xi|)=\bigo(|\xi|^{-1})$ and $R'(|\xi|)=\bigo(|\xi|^{-2})$. Therefore Lemma \ref{lem:osc_W11} gives
    \begin{align*}
        U(t)\psi_\nu=U_0(t)\psi_\nu+B_\nu(t),\quad \text{with }\left\|B_\nu(t)\right\|_{L^\infty_x}\lesssim \nu^{-1}\la t\ra^{-\frac{1}{2}}.
    \end{align*}
    Using $\left| |z+w|^4 -|z|^4\right|\lesssim |z|^3 |w|+|w|^4 $, and H\"older inequality on $d\mu_\nu=\nu|a(x)|dx$, we have
    \begin{align*}
        &\left|\nu \int_0^\infty \int a(x)\left(\big|U(t)\psi_\nu\big|^4 - \big|U_0(t)\psi_\nu\big|^4\right)dxdt\right|\\
        &\qquad \qquad \qquad \quad \lesssim \int_0^\infty \left\|U_0(t)\psi_\nu\right\|_{L^4_{d\mu_\nu}}^3 \|B_\nu(t)\|_{L^4_{d\mu_\nu}}dt+\int_0^\infty \|B_\nu(t)\|_{L^4_{d\mu_\nu}}^4dt.
    \end{align*}
    The second term is immediately bounded by 
    \[
    \int_0^\infty \|B_\nu(t)\|_{L^4_{d\mu_\nu}}^4dt\lesssim_\phi \nu \|a\|_{L^1} \nu^{-4}\int_0^\infty \la t\ra ^{-2}\lesssim_\phi \|a\|_{L^1}\nu^{-3}.  
    \]
    For the first term, H\"older in time then gives
    \begin{align*}
         \int_0^\infty \left\|U_0(t)\psi_\nu\right\|_{L^4_{d\mu_\nu}}^3 \|B_\nu(t)\|_{L^4_{d\mu_\nu}}dt\lesssim_\phi \nu^{-\frac{3}{4}}\left(  \int_0^\infty \left\|U_0(t)\psi_\nu\right\|_{L^4_{d\mu_\nu}}^4 dt\right)^\frac{3}{4}.
    \end{align*}
    By Galilean invariance of the free flow,
    \[
    U_0(t)\psi_\nu(x)= e^{ix\nu-it\nu^2}U_0(t)\phi(x-2\nu t),
    \]
    we rewrite
    \begin{align*}
         \int_0^\infty \left\|U_0(t)\psi_\nu\right\|_{L^4_{d\mu_\nu}}^4 dt=\frac{1}{2}\int_0^\infty \int |a(x)|\big|U_0\left(\frac{s}{2\nu}\right)\phi(x-s)\big|^4 dxdt
    \end{align*}
    By Lemma \ref{lem:free_packet_localization}, we bound
    \begin{align*}
        \int_0^\infty \left\|U_0(t)\psi_\nu\right\|_{L^4_{d\mu_\nu}}^4 dt\lesssim _\phi \|a\|_{L^1}.
    \end{align*}
    Combining the previous estimates, we obtain 
    \begin{align*}
        \left|\nu \int_0^\infty \int a(x)\left(\big|U(t)\psi_\nu\big|^4 - \big|U_0(t)\psi_\nu\big|^4\right)dxdt\right|\lesssim \nu^{-\frac{3}{4}};
    \end{align*}
    hence
    \begin{equation*}
        \lim_{\nu\to \infty}\nu \int_0^\infty \int a(x)\left(\big|U(t)\psi_\nu\big|^4 - \big|U_0(t)\psi_\nu\big|^4\right)dxdt=0.
    \end{equation*}
    With assumption \ref{eq:assumption_unique}, we have
    \begin{equation}
        \lim_{\nu\to\infty }\int_0^\infty \int a(x) \big|U_0\left(\frac{s}{2\nu}\right)\phi(x-s)\big|^4 dxds=0.
    \end{equation}
    The uniform control \eqref{eq:free_packet_uniform_integrability} and dominated convergence theorem then gives, for any $\phi\in C_c^\infty (\R_+)$
    \begin{equation}
        0=\int_0^\infty \int a(x)|\phi(x-s)|^4dxds=\int _0^\infty |\phi(y)|^4 \left(\int _y^\infty a(s)ds\right) dy.
    \end{equation}
    As a result, we have
    \[
    \int_y^\infty a(s)ds =0,\quad \text{for every }y\in \R_+,
    \]
    since the integral is an absolutely continuous function on $\R_+$ from $a\in L^1$. Differentiating this, we have $a=0$ almost everywhere on the positive half line. $a$ on the negative half line can be recovered similarly with $e^{-ix\nu}\phi$ for $\nu\geq 10$ and $\phi\in C_c^\infty((-\infty,0))$. The same computation can be done to verify that $a=0$ almost everywhere on the whole line.
\end{proof}
We now finish the proof of Theorem \ref{thm:uniqueness}.
\begin{proof}[Proof of Theorem \ref{thm:uniqueness}]
Assume that $S_{q,a}=S_{\widetilde q,b}$ near the origin.  Proposition \ref{prop:recover_q} gives $q=\widetilde q$.  We henceforth denote this common value by $q$ and suppress it from the notation.  Proposition
\ref{prop:difference_third_order} then yields
\[
    \int_0^\infty\!\int_{\mathbb R}
       (a-b)(x)|U_q(t)\psi(x)|^4\,dx\,dt=0,
    \qquad \psi\in\mathcal S(\mathbb R).
\]
Applying Proposition \ref{prop:recover_a_delta_flow} to $a-b\in L^1$ gives
$a=b$ almost everywhere.
\end{proof}


\section{Stability for the modified scattering map}\label{sec:stability}
Throughout this section, $q, \tilde q>0, \lambda\neq 0$ and $\beta \in (0,\frac{1}{12})$ are fixed. We compare the modified scattering maps corresponding to two parameter pairs $(q,a)$ and $(\tilde q, b)$.  We use $C_{a,b,q,\tilde q}$ to denote a positive constant depending only on the fixed parameters $q,\tilde q, \lambda,\beta$ and the admissibility norms $\|a\|_{\mathrm{ad}}$ and $\|b\|_{\mathrm{ad}}$. Its value may change from line to line. We also set
\[
\|\psi\|_X:=\|\psi\|_\Sigma+\|\psi\|_{L^1}.
\]
First, we shall verify that \eqref{def:SaSb} is finite. Indeed, the matrices $B_q$, $B_{\tilde q}$ and the diagonal phase matrices $\mathcal{D}_{q,a}$, $\mathcal{D}_{\tilde q,b}$ are pointwise unitary. Consequently,
\[
\big\|\vec g_{q,a}(t)\big\|_{L^2\times L^2}=\big\|\vec w_{q,a}\big\|_{L^2\times L^2}=\sqrt{2}\|u_0\|_{L^2}.
\]
The same identity holds for $\vec g_{\tilde q,b}(t)$. Since $\vec g_{q,a}(t)$ converges to $S_{q,a}(u_0)$, Fatou's Lemma implies 
\[
\|S_{q,a}(u_0)\|_{L^2\times L^2}+\|S_{\tilde q,b}(u_0)\|_{L^2\times L^2}\leq 2\sqrt{2}\|u_0\|_{L^2}.
\]
Hence $\|S_{q,a}-S_{\tilde q,b}\|<\infty$. 


\subsection{Lipschitz stability of the point-interaction strength}
\label{subsec:stability_delta}
We begin with the linear part of the modified scattering map.  Fix a nonzero, nonnegative function $\psi_0\in C_c^\infty((0,\infty))$, and set
\[
    \vec m_{q,\widetilde q}
    :=\begin{bmatrix}
        0\\[0.3em]
        \mathbf 1_{(0,\infty)}(\xi)
        \bigl(T_q(\xi)-T_{\widetilde q}(\xi)\bigr)
        \widehat\psi_0(-\xi)
      \end{bmatrix}.
\]
\begin{proposition}\label{prop:stability_q}
For $q,\widetilde q>0$ and admissible $a,b$,
\[
    |q-\widetilde q|
    \leq C_{q,\widetilde q}
    \|S_{q,a}-S_{\widetilde q,b}\|.
\]
\end{proposition}
\begin{proof}
By \eqref{eq:Lq_second_component},
\[
    \big\langle(L_q-L_{\widetilde q})\psi_0,
       \vec m_{q,\widetilde q}\big\rangle
    =\|\vec m_{q,\widetilde q}\|_{L^2\times L^2}^2.
\]
For small $\varepsilon>0$, Lemma \ref{lem:linearization_scattering_map} established
\begin{equation}\label{eq:L_qa}
S_{q,a}(\varepsilon\psi_0)=\varepsilon L_q\psi_0 + \mathcal{R}_{q,a}(\varepsilon,\psi),\qquad \big\| \mathcal{R}_{q,a}(\varepsilon,\psi)\big\|_{L^\infty\times L^\infty}\leq C_{q,a}\varepsilon^3 \|\psi\|_\Sigma^3,
\end{equation}
and analogously for $(\tilde q,b)$. Pairing the two expansions in \eqref{eq:L_qa} with $\vec m_{q,\widetilde q}$, dividing by $\varepsilon$, and letting $\varepsilon\to0$ gives
\[
    \|\vec m_{q,\widetilde q}\|_{L^2\times L^2}^2
    \leq \|S_{q,a}-S_{\widetilde q,b}\|\,\|\psi_0\|_\Sigma \|\vec m_{q,\widetilde q}\|_{L^2\times L^2}.
\]
Here the cubic remainders vanish because
$\|\vec m_{q,\tilde q}\|_{L^1\times L^1}\lesssim\|\psi_0\|_\Sigma$, as $T_q-T_{\tilde q}$ belong to $L^2$. If $q\neq\widetilde q$, we may
divide by the $\|\vec m_{q,\tilde q}\|_{L^2\times L^2}$ norm.  Moreover,
\[
\begin{split}
    |T_q(\xi)-T_{\widetilde q}(\xi)|
    &=\frac{2|\xi|\,|q-\widetilde q|}
      {(4\xi^2+q^2)^{\frac{1}{2}}(4\xi^2+\widetilde q^2)^{\frac{1}{2}}}\\
    &\geq c_{q,\widetilde q}|q-\widetilde q|
      \frac{|\xi|}{\langle\xi\rangle^2}.
\end{split}
\]
Consequently,
\[
    \|\vec m_{q,\widetilde q}\|_{L^2\times L^2}
    \geq c_{q,\widetilde q}|q-\widetilde q|
       \left\|\frac{\xi}{\langle\xi\rangle^2}
       \widehat\psi_0(\xi)\right\|_{L^2(\mathbb R_-)}.
\]
The last norm is positive because $\widehat\psi_0(0)= \int\psi_0>0$. Combining the preceding estimates proves
the proposition. The case $q=\widetilde q$ is immediate.
\end{proof}


\subsection{Dependence on the delta strength}
\label{subsec:parameter_dependence}

Fix a compact interval $I=[q_-,q_+]\subset(0,\infty)$ and an admissible coefficient $b$. All constants in this subsection are uniform for $\rho\in I$.  We write
\[
    u_\rho:=u_{\rho,b},\qquad
    w_\rho:=w_{\rho,b},\qquad
    z_\rho:=\partial_\rho w_\rho,\qquad
    v_\rho:=\partial_\rho u_\rho,
\]
and set $\alpha:=\varepsilon\|\psi\|_X$.  For complete rigor, the argument may first be carried out for the difference quotients in $\rho$; all estimates below are uniform in the increment, and passage to the limit gives the asserted
derivatives. We use derivative notation for simplicity.

\medskip

\paragraph{\bf Linear parameter estimates} Set
\begin{equation}\label{eq:def_d_rho}
    d_\rho(\xi):=\partial_\rho T_\rho(\xi)
    =\partial_\rho R_\rho(\xi)
    =\frac{2i\xi}{(2i\xi-\rho)^2}.
\end{equation}
Direct computation gives
\begin{equation}\label{eq:d_rho_bounds}
    \|d_\rho\|_{L^\infty_\xi}=\frac1{2\rho}\leq\frac1{2q_-},
    \qquad
    \|\partial_\xi d_\rho\|_{L^\infty_\xi}\leq\frac2{q_-^2},
    \qquad
    \|\xi d_\rho(\xi)\|_{L^\infty_r}\leq\frac12.
\end{equation}
Thus
\[
    \sup_{\xi\in\mathbb R}|\partial_\rho T_\rho(\xi)|
    =\sup_{\xi\in\mathbb R}|\partial_\rho R_\rho(\xi)|
    =\frac1{2\rho}.
\]
In particular,
\begin{equation}\label{eq:partial_B_bound}
    \sup_{\rho\in I}\|\partial_\rho B_\rho\|_{L^\infty_\xi}
    +\sup_{\rho\in I}\|\partial_\rho B_\rho^*\|_{L^\infty_\xi}
    \lesssim_I1.
\end{equation}

We record the following useful estimates on $\partial_\rho \dft_\rho$ and $\partial_\rho\dft_\rho^{-1}$.
\begin{lemma}\label{lem:dft_partial_q_est}
Uniformly for $\rho\in I$,
\begin{equation}\label{eq:partial_q_dft}
    \|\partial_\rho\dft_\rho \phi\|_{L^2}\lesssim_I\|\phi\|_{L^2},\qquad 
    \|\partial_\rho\dft_\rho^{-1}\psi\|_{L^2} \lesssim_I\|\psi\|_{L^2},
\end{equation}
\begin{equation}\label{eq:partial_q_dft_H1}
    \|\partial_\rho\dft_\rho \phi\|_{H^1}
      \lesssim_I\|\langle x\rangle \phi\|_{L^2},
\end{equation}
\begin{equation}\label{eq:partial_q_dft_Linfty}
     \|\partial_\rho\dft_\rho \phi\|_{L^\infty}
      \lesssim_I\|\phi\|_{L^1},
\end{equation}
\begin{equation}\label{eq:partial_q_dft_weight}
    \|\xi\partial_\rho\dft_\rho \phi\|_{L^2}
      \lesssim\|\phi\|_{L^2}.
\end{equation}
Moreover, $(\partial_\rho\dft_\rho \phi)(0)=0$.
\end{lemma}
\begin{proof}
Let
\[
    G_\phi(r):=\frac1{\sqrt{2\pi}}
       \int_{\mathbb R}e^{-i|x|r}\phi(x)\,dx,
    \qquad r\geq0.
\]
Lemma \ref{lem:relation} gives
\[
    \partial_\rho\dft_\rho \phi(\xi)
      =\overline{d_\rho(|\xi|)}\frac1{\sqrt{2\pi}}
       \int_{\mathbb R}e^{-i|x||\xi|}\phi(x)\,dx,.
\]
Splitting both $x$ and $\xi$ into half-lines and applying Plancherel yields
\[
    \left\| \int_{\mathbb R}e^{-i|x||\xi|}\phi(x)\,dx\right\|_{L^2_\xi}\lesssim\|\phi\|_{L^2},
    \qquad
    \left\|\partial_r\int_{\mathbb R}e^{-i|x|r}\phi(x)\,dx\bigg|_{r=|\xi|}\right\|_{L^2_\xi}\lesssim\|x\phi\|_{L^2}.
\]
Estimates \eqref{eq:partial_q_dft}--\eqref{eq:partial_q_dft_weight} now follow
from \eqref{eq:d_rho_bounds}; the inverse estimate is proved in exactly the same way from \eqref{eq:dft_inverse_flat_relation}.  Finally, $d_\rho(0)=0$,
which proves the assertion at the origin.
\end{proof}

Next, we establish the decomposition and estimates related to $\partial_\rho V_\rho(t)$ and $\partial_\rho V_\rho(t)^{-1}$. 
\begin{proposition}\label{prop:partial_V_est}
    For $\psi,\phi\in H^1(\R)$, $t>0$ and $0<\tilde q<\rho<q$, we have
    \begin{equation}
        \big\|\partial_\rho V_\rho(t)\phi(x)-d_\rho(|x|)\phi(x)-d_\rho(|x|)\phi(-x)\big\|_{L^\infty}\lesssim _{I} t^{-\frac{1}{4}}\|\phi\|_{H^1},
    \end{equation}
    and 
    \begin{equation}
        \big\|\partial_\rho V_\rho(t)^{-1}\psi(x)-\overline{d_\rho(|x|)}\psi(x)-\overline{d_\rho(|x|)}\psi(-x)\big\|_{L^\infty}\lesssim _{I} t^{-\frac{1}{4}}\|\psi\|_{H^1}.
    \end{equation}
    If $\phi(0)=0$, then 
    \begin{equation}
        \big\|\partial_\rho V_\rho(t)\phi\big\|_{H^1}\lesssim_{I}\|\phi\|_{H^1},
    \end{equation}
    while 
    \begin{equation}
        \big\|\partial_\rho V_\rho(t)^{-1}\psi\big\|_{H^1}\lesssim_{q,\tilde q} \,t^\frac{1}{2} |\psi(0)|+\|\psi\|_{H^1}.
    \end{equation}
\end{proposition}
\begin{proof}
Differentiate the kernel formulas used in the proofs of Proposition \ref{prop:asy_V} and Proposition \ref{prop:homo_sobolev_est}. In the notation of \cite[(3.7)--(3.9), (3.20)--(3.22)]{MMS}, the parameter occurs only through
$T_\rho$ and $R_\rho$; differentiation therefore replaces either coefficient by $d_\rho$. The $L^\infty$ argument is unchanged because $d_\rho$ and $d_\rho'$ obey \eqref{eq:d_rho_bounds}. The integrations by parts used for the
$H^1$ estimates are also unchanged. The boundary terms are controlled by $d_\rho(0)=0$, and the possible trace term for the inverse is estimated as in Proposition \ref{prop:homo_sobolev_est}. All constants are uniform on $I$.
\end{proof}

\medskip

\paragraph{\bf The differentiated profile} Since $w_\rho(t,0)=0$ for every $\rho$, one has
\begin{equation}\label{eq:z_zero_frequency}
    z_\rho(t,0)=0.
\end{equation}
\begin{lemma}\label{lem:quant_z_v_small} For $\varepsilon$ sufficiently small, 
\begin{equation}\label{eq:quant_z_v_small}
    \sup_{\rho\in I}\sup_{0\leq t\leq 1}
    \bigl(\|z_\rho(t)\|_{H^1}+\|v_\rho(t)\|_{L^2}\bigr)
    \lesssim_{I,b}\alpha.
\end{equation}
Consequently,
\begin{equation}\label{eq:partial_small_time_remainder}
\begin{split}
    \sup_{\rho\in I}
    \left\|\partial_\rho\left\{
      B_\rho^*\int_0^1
      \overrightarrow{\dft_\rho U_\rho(-t)
       [(1+b)|u_\rho|^2u_\rho]}
      \,dt\right\}\right\|_{L^\infty\times L^\infty}
    \lesssim_{I,b}\alpha^3.
\end{split}
\end{equation}
\end{lemma}
\begin{proof}
We use throughout the uniform small-time estimates proved in Section \ref{sec:direct}
\begin{equation}\label{eq:uniform_small_time_direct}
    \sup_{\rho\in I}\sup_{0\leq t\leq1}
    \left(\|u_\rho(t)\|_{H^1}+\|xu_\rho(t)\|_{L^2} +\|w_\rho(t)\|_{H^1}\right)\lesssim_{I,b}\alpha.
\end{equation}
Since
\[
    u_\rho(t)= \dft_\rho^{-1} \bigl(e^{-it\xi^2}w_\rho(t)\bigr),
\]
differentiation in $\rho$ gives
\begin{equation}\label{eq:v_transform_identity}
    v_\rho(t)=
    (\partial_\rho\dft_\rho^{-1})\bigl(e^{-it\xi^2}w_\rho(t)\bigr)+\dft_\rho^{-1}\bigl(e^{-it\xi^2}z_\rho(t)\bigr).
\end{equation}
Consequently,
\begin{equation}\label{eq:v_L2_by_z}
    \|v_\rho(t)\|_{L^2}\lesssim_I\alpha+\|z_\rho(t)\|_{L^2}.
\end{equation}
Differentiating the profile Duhamel formula yields
\begin{equation}\label{eq:z_small_Duhamel}
\begin{split}
    z_\rho(t)=\varepsilon(\partial_\rho\dft_\rho)\psi&-i\lambda
    \int_0^t e^{is\xi^2} (\partial_\rho\dft_\rho (1+b)|u_\rho|^2 u_\rho(s) ds\\
    &\,\,\,\,\,\,-i\lambda \int_0^t e^{is\xi^2} \dft_\rho (1+b)\big( 2|u_\rho|^2 v_\rho+ u_\rho^2 \overline {v_\rho}\big)(s)ds.
\end{split}
\end{equation}
By \eqref{eq:uniform_small_time_direct},
\[
    \|(1+b)|u_\rho|^2u_\rho(s)\|_{L^2} \lesssim_b\alpha^3,\qquad
    \|(1+b)\left( 2|u_\rho|^2v_\rho+u_\rho^2\overline{v_\rho}\right)(s)\|_{L^2}\lesssim_b\alpha^2\|v_\rho(s)\|_{L^2}.
\]
Therefore, Lemma \ref{lem:dft_partial_q_est} and
\eqref{eq:z_small_Duhamel} give
\[
\begin{split}
    \|z_\rho(t)\|_{L^2} \lesssim_{I,b}{}\alpha+\alpha^3 +\alpha^2\int_0^t\|v_\rho(s)\|_{L^2}\,ds.
\end{split}
\]
Using \eqref{eq:v_L2_by_z} and Gronwall's inequality, we obtain
\begin{equation}\label{eq:z_v_small_L2}
    \sup_{\rho\in I}\sup_{0\leq t\leq1} \left(  \|z_\rho(t)\|_{L^2}  +\|v_\rho(t)\|_{L^2} \right) \lesssim_{I,b}\alpha.
\end{equation}
We next estimate $\partial_\xi z_\rho$. We first record that, for
$f\in H^1$ with $xf\in L^2$,
\begin{equation}\label{eq:x_f_squared}
    \|x|f|^2\|_{L^\infty}
    \lesssim
    \|f\|_{L^2}^2
    +\|xf\|_{L^2}\|\partial_xf\|_{L^2}.
\end{equation}
It follows from \eqref{eq:uniform_small_time_direct},
\eqref{eq:z_v_small_L2}, and \eqref{eq:x_f_squared} that
\begin{equation}\label{eq:G0_G1_weighted_small}
    \|\langle x\rangle (1+b)|u_\rho|^2u_\rho(s)\|_{L^2}+
    \|\langle x\rangle(1+b)\left( 2|u_\rho|^2v_\rho+u_\rho^2\overline {v_\rho}\right(s)\|_{L^2} \lesssim_{I,b}\alpha^3,\qquad 0\leq s\leq1.
\end{equation}
The parameter-dependent distorted Fourier estimates imply
\begin{equation}\label{eq:parameter_transform_small_H1}
\begin{split}
    \left\|
        e^{is\xi^2}
        (\partial_\rho\dft_\rho)f
    \right\|_{H^1_\xi}
    \lesssim_I
    \|\langle x\rangle f\|_{L^2}
    +s\|f\|_{L^2},
    \qquad 0\leq s\leq1.
\end{split}
\end{equation}
Indeed, when $\partial_\xi$ hits the phase, the resulting term is
controlled by
\[
    \|\xi(\partial_\rho\dft_\rho)f\|_{L^2}
    \lesssim_I\|f\|_{L^2},
\]
whereas the remaining term is bounded by
\[
    \|\partial_\xi(\partial_\rho\dft_\rho)f\|_{L^2}
    \lesssim_I\|\langle x\rangle f\|_{L^2}.
\]
Thus, by \eqref{eq:G0_G1_weighted_small}, every term in
$\partial_\xi z_\rho$ is bounded by $C_{I,b}\alpha^3$, except when $\partial_\xi$ falls on the oscillation
\begin{equation}\label{eq:J_rho_definition}
    \mathcal J_\rho(t):=\int_0^t  2is\xi e^{is\xi^2} \dft_\rho (1+b)\left( 2|u_\rho|^2 v_\rho+u_\rho^2 \overline{v_\rho}\right)(s)ds.
\end{equation}
The local smoothing estimate of Lemma
\ref{lem:dft_local_smoothing} is valid, with the same constant, on every time interval; hence it may be applied on $[0,t]$. More specifically, we split \eqref{eq:J_rho_definition} according to $1+b$.  For the localized part, repeating the low-/high-frequency decomposition in the
proof of Lemma \ref{lem:H1_forcing_partial}, now on $[0,t]$, and using the small-time bounds for $u_{\rho,b}$ and $v_{\rho,b}$ gives
\[ \left\|
    \int_0^t 2is\xi e^{is\xi^2}\dft_\rho\left[  b\left( 2|u_{\rho,b}|^2v_{\rho,b} +u_{\rho,b}^2\overline{v_{\rho,b}}  \right) \right](s)\,ds
\right\|_{L^2_\xi}
\lesssim_{q,\tilde q,b}\varepsilon^2\|\psi\|_\Sigma^2
\left( \varepsilon\|\psi\|_\Sigma +\sup_{0\le s\le t}\|z_{\rho,b}(s)\|_{H^1}\right).
\]
Indeed, after integration by parts in the flat high-frequency term, the factorization identities for $s\partial_xu_{\rho,b}$ and $s\partial_xv_{\rho,b}$ reduce the resulting expressions to precisely the terms estimated in Lemma \ref{lem:H1_forcing_partial}.  Since $0\le s\le1$, no time growth occurs.

For the homogeneous part, Lemma \ref{lem:dft_est} gives
\[
\left\| \int_0^t 2is\xi e^{is\xi^2}\dft_\rho\left( 2|u_{\rho,b}|^2v_{\rho,b} +u_{\rho,b}^2\overline{v_{\rho,b}}\right)(s)\,ds
\right\|_{L^2_\xi}
\lesssim_{q,\tilde q,b}
\varepsilon^2\|\psi\|_\Sigma^2
\left(\varepsilon\|\psi\|_\Sigma+ \sup_{0\le s\le t}\|z_{\rho,b}(s)\|_{H^1}
\right).
\]
Consequently,
\begin{equation}\label{eq:J_rho_bound}
    \|\mathcal J_\rho(t)\|_{L^2_\xi}
    \lesssim_{q,\tilde q,b}\varepsilon^2\|\psi\|_\Sigma^2
    \left(\varepsilon\|\psi\|_\Sigma + \sup_{0\le s\le t}\|z_{\rho,b}(s)\|_{H^1} \right),  \qquad 0\le t\le1.
\end{equation}
Returning to \eqref{eq:z_small_Duhamel}, we conclude that
\[
    Z(t):=\sup_{\rho\in I}\sup_{0\leq s\leq t}
    \|z_\rho(s)\|_{H^1},\qquad Z(t) \leq C_{I,b}\alpha +C_{I,b}\alpha^2Z(t).
\]
For $\alpha$ sufficiently small, the last term is absorbed, which proves \eqref{eq:quant_z_v_small}. With uniform boundedness on $B_\rho^*$ and $\partial_\rho B_\rho^*$ from \eqref{eq:d_rho_bounds}, \eqref{eq:partial_small_time_remainder} is trivial then $\partial_\rho$ hits $B_\rho^*$. Otherwise, we bound using \eqref{eq:partial_q_dft} and \eqref{eq:quant_z_v_small}
\begin{align*}
    &\int_0^1 \left\|\left(\partial_\rho\dft_\rho \right) (1+b)|u_{\rho,b}(t)|^2 u_{\rho,b}(t)\right\|_{L^2} dt+C_b\int_0^1 \big\|u_{\rho,b}(t)\big\|_{L^\infty}^2 \big\|v_{\rho,b}(t)\big\|_{L^2} dt\lesssim_{q,\tilde q,b} \varepsilon^3 \|\psi\|_\Sigma^3,
\end{align*}
and analogously for the second coordinate. Therefore, we have the uniform estimate \eqref{eq:partial_small_time_remainder}. 
\end{proof}

\medskip

Fix once and for all
\begin{equation}\label{eq:choice_gamma}
    0<\gamma<\frac14-3\beta.
\end{equation}
The direct estimates, uniformly for $\rho\in I$, give
\begin{equation}\label{eq:uniform_direct_bounds_rho}
    \|w_\rho(t)\|_{L^\infty} +t^{-\beta}\|w_\rho(t)\|_{H^1}+t^{\frac{1}{2}}\|u_\rho(t)\|_{L^\infty} \lesssim_{I,b}\alpha, \qquad t\geq1.
\end{equation}
The differentiated localized forcing is
\begin{equation}\label{eq:expression_partial_N}
\begin{split}
    \partial_\rho\mathcal N_{\rho,b}(t)=e^{it\xi^2}(\partial_\rho\dft_\rho) [b|u_\rho|^2u_\rho] +e^{it\xi^2}\dft_\rho[b(2|u_\rho|^2v_\rho+u_\rho^2\overline{v_\rho})].
\end{split}
\end{equation}

\begin{lemma}\label{lem:H1_forcing_partial}
Suppose on $[1,T]$ that
\begin{equation}\label{eq:z_bootstrap}
    \|z_\rho(t)\|_{L^\infty}
      +t^{-\beta}\|z_\rho(t)\|_{H^1}
    \leq K\alpha t^\gamma
\end{equation}
uniformly for $\rho\in I$.  Then
\begin{align}
    \|v_\rho(t)\|_{L^\infty}
      &\lesssim_I K\alpha t^{-\frac{1}{2}+\gamma},
      \label{eq:v_large_parameter}\\
    \|\partial_\rho\mathcal N_{\rho,b}(t)\|_{L^\infty}
      &\lesssim_{I,b}K^3\alpha^3t^{-\frac{3}{2}+\gamma},
      \label{eq:partial_N_Linfty}\\
    \left\|\int_1^t\partial_\rho\mathcal N_{\rho,b}(s)\,ds
       \right\|_{H^1}
      &\lesssim_{I,b}K^3\alpha^3t^{\beta+\gamma}.
      \label{eq:partial_N_H1}
\end{align}
\end{lemma}
\begin{proof}
The factorization
\begin{equation}\label{eq:v_factor}
     v_\rho=M(t)D(t)\big[(\partial_\rho V_\rho(t))w_\rho  +V_\rho(t)z_\rho\big]
\end{equation}
together with Proposition \ref{prop:partial_V_est} gives
\eqref{eq:v_large_parameter}.  Formula \eqref{eq:expression_partial_N} and the
$L^1\to L^\infty$ estimates for $\dft_\rho$ and
$\partial_\rho\dft_\rho$ then give \eqref{eq:partial_N_Linfty}.

For \eqref{eq:partial_N_H1}, follow the proof of Lemma
\ref{lem:H1_forcing_a}.  Terms for which $\partial_\xi$ does not hit $e^{it\xi^2}$ are integrable by \eqref{eq:partial_q_dft_H1} and \eqref{eq:dft_est_2}.  When it hits the oscillation, local smoothing bounds the first term in
\eqref{eq:expression_partial_N} by
\[
    \|tb|u_\rho|^2u_\rho\|_{L_x^1L_t^2([1,T])}\lesssim_b\alpha^3(\log T)^{\frac{1}{2}},
\]
and the low-frequency and reflected pieces containing $v_\rho$ by
\[
    \|tb(2|u_\rho|^2v_\rho+u_\rho^2\overline{v_\rho})  \|_{L_x^1L_t^2([1,T])} \lesssim_{I,b} K^3\alpha^3T^\gamma(\log T)^{\frac{1}{2}}.
\]
For the remaining flat high-frequency term, integrate by parts in $x$. With
\begin{equation*}
    \begin{split}
    \partial_x(2|u|^2v+u^2\bar v) =4\operatorname{Re}(\bar u\,u_x)v+2|u|^2v_x+2uu_x\bar v+u^2\overline{v_x},
\end{split}
\end{equation*}
the factorization identities used in the proof of Lemma
\ref{lem:quant_z_v_small} reduce every term to the same estimates as in Lemma \ref{lem:H1_forcing_a}, with one factor $u_\rho$ replaced by $v_\rho$ or one factor $V_\rho w_\rho$ replaced by $V_\rho z_\rho+(\partial_\rho V_\rho)w_\rho$. Proposition \ref{prop:partial_V_est} and \eqref{eq:z_bootstrap} therefore give $C_{I,b}K^3\alpha^3T^{\beta+\gamma}$. Since logarithmic factors are absorbed by $T^\beta$, this proves \eqref{eq:partial_N_H1}.
\end{proof}

\begin{proposition}\label{prop:quant_z_big}
If $\alpha$ is sufficiently small depending on $I,b,\beta$, and $\gamma$, then
\begin{equation}\label{eq:z_global_parameter}
    \sup_{\rho\in I}
    \left(\|z_\rho(t)\|_{L^\infty}
      +t^{-\beta}\|z_\rho(t)\|_{H^1}\right)
    \lesssim_{I,b,\beta,\gamma}\alpha t^\gamma,
    \qquad t\geq1.
\end{equation}
Consequently,
\begin{equation}\label{eq:v_global_parameter}
    \sup_{\rho\in I}\|v_\rho(t)\|_{L^\infty}
    \lesssim_{I,b,\beta,\gamma}\alpha t^{-\frac{1}{2}+\gamma}.
\end{equation}
\end{proposition}
\begin{proof}
    By Lemma \ref{lem:quant_z_v_small} and Sobolev embedding, 
    \[
    \sup_{\rho\in I} \left( \|z_\rho(1)\|_{L^\infty}+\|z_\rho(1)\|_{H^1}\right)\lesssim_{I,b}\alpha.
    \]
    Fix $T>1$. We close a bootstrap on $[1,T]$. Namely, for a constant $C_0$, assume that
    \begin{equation}\label{eq:z_parameter_bootstrap}
        \sup_{\rho\in I} \left( \|z_\rho(t)\|_{L^\infty}+t^{-\beta}\|z_\rho(t)\|_{H^1}\right)\leq C_0\alpha t^\gamma,\quad 1\leq t\leq T.
    \end{equation}
    Then Lemma \ref{lem:H1_forcing_partial} implies
    \begin{equation}\label{eq:v_bootstrap_decay}
         \sup_{\rho\in I}\|v_\rho(t)\|_{L^\infty}\lesssim _I C_0\alpha t^{-\frac{1}{2}+\gamma}. 
    \end{equation}
   Differentiating the profile equation with respect to $\rho$, we obtain
\begin{align*}
    z_\rho(t)=z_\rho(1)-\frac{i\lambda}{2}\int_1^t s^{-1}
    \partial_\rho \left\{  V_\rho(s)^{-1} \big(|V_\rho(s)w_\rho(s)|^2V_\rho(s)w_\rho(s) \big)\right\}\,ds-i\lambda\int_1^t\partial_\rho\mathcal N_{\rho,b}(s)ds.
\end{align*}
For the homogeneous term, we differentiate both $V_\rho^{-1}$ and
$V_\rho w_\rho$.  The product estimates used in Lemma
\ref{lem:close_boot}, together with Proposition
\ref{prop:partial_V_est}, give
\begin{align*}
    &s^{-1}\left\|\partial_\rho \left\{V_\rho(s)^{-1} \big(
  |V_\rho(s)w_\rho(s)|^2V_\rho(s)w_\rho(s)\big)\right\}
    \right\|_{H^1}\,\\
    &\qquad \lesssim_{I,b} \alpha^3\left( s^{-1+\beta}+s^{-\frac{5}{4}+3\beta}
    \right)+\alpha^2\left(s^{-1}\|z_\rho(s)\|_{H^1} +s^{-1+\beta}\|z_\rho(s)\|_{L^\infty}+s^{-\frac{5}{4}+2\beta}\|z_\rho(s)\|_{H^1}\right).
\end{align*}
The first line on the right-hand side contains the terms in which the parameter derivative falls only on the operators, whereas the second line contains the terms involving $z_\rho$.  Under \eqref{eq:z_parameter_bootstrap}, and combining with Lemma \ref{lem:H1_forcing_a}, we obtain
\begin{equation}\label{eq:z_H1_bootstrap_improvement}
    \sup_{\rho\in I}\|z_\rho(t)\|_{H^1} \leq C_{I,b}\alpha
    +C_{I,b,\beta,\gamma,C_0}\alpha^3t^{\beta+\gamma}.
\end{equation}

We next estimate the $L^\infty$ norm.  The corresponding pointwise
estimate for the homogeneous term is
\begin{align*}
    &s^{-1} \left\| \partial_\rho\left\{V_\rho(s)^{-1}
    \big(|V_\rho(s)w_\rho(s)|^2V_\rho(s)w_\rho(s)\big)\right\}
    \right\|_{L^\infty}\\
    &\quad\lesssim_{I,b}\alpha^3\left(s^{-1} +s^{-\frac{5}{4}+\beta}\right)+\alpha^2\left(s^{-1}\|z_\rho(s)\|_{L^\infty}+s^{-\frac{5}{4}+\beta}\|z_\rho(s)\|_{L^\infty}+s^{-\frac{5}{4}}\|z_\rho(s)\|_{H^1}\right).
\end{align*}
The term $\alpha^3s^{-1}$ is produced when the parameter derivative
falls on the $\rho$-dependent scattering coefficients in the leading
cubic expression.  Using \eqref{eq:z_parameter_bootstrap} and
$\beta+\gamma<\frac{1}{4}$, we obtain
\begin{align*}
    \int_1^t s^{-1}\left\|\partial_\rho\left\{ V_\rho(s)^{-1}
    \big(|V_\rho(s)w_\rho(s)|^2V_\rho(s)w_\rho(s)\big)\right\}
    \right\|_{L^\infty}\,ds\lesssim_{I,b,\beta,\gamma,C_0}
    \alpha^3\bigl(\log t+t^\gamma\bigr).
\end{align*}
On the other hand, \eqref{eq:expression_partial_N}, the decay estimate for
$u_\rho$, and \eqref{eq:v_bootstrap_decay} give
\begin{align*}
    \|\partial_\rho\mathcal N_{\rho,b}(s)\|_{L^\infty}\lesssim_b
    \|u_\rho(s)\|_{L^\infty}^3 +\|u_\rho(s)\|_{L^\infty}^2
      \|v_\rho(s)\|_{L^\infty}\lesssim_{I,b,\beta,C_0}
    \alpha^3\left(s^{-\frac{3}{2}} +s^{-\frac{3}{2}+\gamma}\right),
\end{align*}
which is integrable.  Consequently,
\begin{equation}\label{eq:z_Linfty_bootstrap_improvement}
    \sup_{\rho\in I}\|z_\rho(t)\|_{L^\infty}
    \leq
    C_{I,b}\alpha
    +C_{I,b,\beta,\gamma,C_0}
      \alpha^3 t^\gamma.
\end{equation}
Combining \eqref{eq:z_H1_bootstrap_improvement} and
\eqref{eq:z_Linfty_bootstrap_improvement}, we conclude that
\begin{align*}
    \sup_{\rho\in I}
    \left(
        \|z_\rho(t)\|_{L^\infty}
        +t^{-\beta}\|z_\rho(t)\|_{H^1}
    \right)
    \leq
    \left(
        C_{I,b}
        +C_{I,b,\beta,\gamma,C_0}\alpha^2
    \right)\alpha t^\gamma,
\end{align*}
which closes the bootstrap argument with large enough $C_0$. 
This proves \eqref{eq:z_global_parameter} and \eqref{eq:v_global_parameter}.
\end{proof}
\begin{remark}
    The positive loss $t^\gamma$ is needed to absorb the long-range $\log(t)$ term. It can be taken as small as needed. 
\end{remark}


\subsection{Stability of inhomogeneity $a$.} In this subsection, we prove the stability of the inhomogeneous coefficients stated in Theorem \ref{thm:stability}. We first record Proposition \ref{prop:difference_third_order} with the dependence on the test function retained. 
\begin{lemma}\label{lem:const_u}
    Let $a,b$ be admissible, $q, \tilde q>0$ and let $\psi\in \mathcal{S}(\R)$. If $\varepsilon\|\psi\|_\Sigma\leq \varepsilon_{q,\tilde q,a,b}$, then 
    \begin{equation}\label{eq:const_cubic}
        \left|\int_0^\infty \int (a-b)(x)\big|U(t)\psi(x)\big|^4 dxdt\right|\lesssim\varepsilon^{-3} \left|\left\la\mapa(\varepsilon\psi)-\mapc(\varepsilon\psi),\vec h_\psi\right\ra\right|+C_{q,a,b}\varepsilon^2 \|\psi\|_{X}^6.
    \end{equation}
\end{lemma}
\begin{proof}
    Retaining the dependence on $\psi$ in the proofs of Lemmas \ref{lem:born_cubic_replacement} and \ref{lem:cancel_homo_error} and Proposition \ref{prop:difference_third_order} gives
    \begin{equation}
        \left|\left \la \,\mapa(\varepsilon\psi)-\mapc(\varepsilon\psi) ,\,\vec{h}_\psi\right\ra+2i\lambda\varepsilon^3 \int_0^\infty \int (a-b)\big| U(t)\psi\big|^4 dxdt\right|\lesssim C_{q,a,b}\varepsilon^5\|\psi\|_X^6.
    \end{equation}
    Here the sixth power arises because each remainder is quintic in the solution or initial data and is paired with one additional copy of $U_q(t)\psi$. Moreover, unitarity of $\dft_q$ implies
    \[
    \big\|\vec h_\psi\big\|_{L^2\times L^2}\lesssim \|\psi\|_{L^2}.
    \]
    Dividing by $\varepsilon^3$ then proves \eqref{eq:const_cubic}.
\end{proof}

\medskip

We fix $\eta\in C_c^\infty((-1,1))$ with $\|\eta\|_{L^4}=1$ and define for $h<1$
\[
\phi_{x_0,h}:=h^{-\frac{1}{4}}\eta\left(\frac{x-x_0}{h}\right).
\]
For $x_0>2h$, this function belongs to $C_c^\infty((0,\infty))$. Direct computations give
\begin{equation}\label{eq:scaling}
    \begin{split}
        \big\|\test\big\|_{L^4}=\|\eta\|_{L^4},\qquad \big\|\test\big\|_{L^2}= h^{\frac{1}{4}}\|\eta\|_{L^2},&\qquad \big\|\test\big\|_{H^\frac{1}{4}}\lesssim_\eta 1,\\
    \big\|\test\big\|_{H^1}=h^{-\frac{3}{4}}\left(\|\eta'\|_{L^2}^2+h^2 \|\eta\|_{L^2}^2 \right)^\frac{1}{2}\lesssim_\eta h^{-\frac{3}{4}}, \quad& \big\|x\test\big\|_{L^2}\lesssim _\eta x_0 h^\frac{1}{4}+h^\frac{5}{4}.
    \end{split}
\end{equation}
Set 
\[
\psi_{\nu,x_0,h}(x):=e^{ix\nu}\test(x).
\]
If $3h\leq x_0\leq h^{-1}$, $\nu h\geq 1$, then 
\begin{equation}\label{eq:scale_X}
    \big \|\testhi\big\|_{L^2}\lesssim_\eta h^\frac{1}{4},\qquad \big \|\testhi\big\|_{X}\lesssim_\eta \nu h^\frac{1}{4}.
\end{equation}
We also recall the one-dimensional Schr\"odinger maximal function estimate of Kenig-Ruiz \cite{KR}; we use the formulation in \cite{KPV}. 
\begin{lemma}\label{lem:maximal_function}
    For $f\in \mathcal{S}(\R)$, define
    \begin{equation*}
        \mathcal{M}_f(x):=\sup_{t\geq 0}\big |U_0(t)f(x)\big|.
    \end{equation*}
    Then 
    \begin{equation*}
        \big\| \mathcal{M}_f\big\|_{L^4}\lesssim \|f\|_{H^\frac{1}{4}}.
    \end{equation*}
\end{lemma}
The main quantitative estimate is the following.
\begin{proposition}\label{prop:primitive_difference}
    Suppose that $3h\leq x_0\leq h^{-1}$, $\nu h^3\geq 1$, $\nu h\geq 1$, $\varepsilon \nu h^\frac{1}{4}\leq \varepsilon_{a,b}$. Then 
    \begin{equation}\label{eq:difference_est_primitive}
        \begin{split}
            \big|\mathcal{A}_+(x_0)-\mathcal{B}_+(x_0)\big|\lesssim_\eta C_{q,a,b}\bigg( &\varepsilon^{-3}\nu \left|\left\la\mapa(\varepsilon\testhi)-\mapc(\varepsilon\testhi),\vec h_{\testhi}\right\ra\right|\\
            &\qquad +\varepsilon^2 \nu^7 h^\frac{3}{2}+\nu^{-\frac{3}{4}}h^{-\frac{5}{4}}+\nu^{-3}h^{-5}+\left(\nu h^3\right)^{-\frac{1}{20}}+h\bigg).
        \end{split}
    \end{equation}
\end{proposition}
\begin{proof}
    Apply Lemma \ref{lem:const_u} to $\testhi$. Using the preceding norm estimates and multiplying by $\nu$, we obtain
    \begin{align*}
        \left|\nu\int_0^\infty \int (a-b)(x)\big|U_q(t)\testhi(x)\big|^4 dxdt\right|\lesssim\varepsilon^{-3}\nu \left|\left\la\mapa(\varepsilon\testhi)-\mapc(\varepsilon\testhi),\vec h_{\testhi}\right\ra\right|+C_{q,a,b}\varepsilon^2 \nu^7 h^\frac{3}{2}.
    \end{align*}
    We repeat the delta/free comparison from the proof of Proposition \ref{prop:recover_a_delta_flow}, retaining the dependence on the scaled packet. Writing 
    \[
    B_{\nu,x_0,h}(t):=U_q(t)\testhi-U_0(t)\testhi,
    \]
    a direct differentiation of the oscillatory amplitudes, together with $x_0\leq h^{-1}$, $\nu h\geq 1$, and the rapid decay of $\hat \eta$, gives
    \[
    \big\| B_{\nu,x_0,h}(t)\big\|_{L^\infty}\lesssim_\eta \la t\ra^{-\frac{1}{2}}\nu^{-1}h^{-\frac{5}{4}}.
    \]
    The argument of Proposition \ref{prop:recover_a_delta_flow}, now using Lemma \ref{lem:maximal_function} and \eqref{eq:scaling}, yields
    \begin{align*}
        \bigg|\int_0^\infty \int (a-&b)(x+s)\big|U_0\left(\frac{s}{2\nu}\right)\test(x)\big|^4 dxds\bigg|\\
        &\lesssim_\eta \varepsilon^{-3}\nu \left|\left\la\mapa(\varepsilon\testhi)-\mapc(\varepsilon\testhi),\vec h_{\testhi}\right\ra\right|+C_{q,a,b}\left(\varepsilon^2 \nu^7 h^\frac{3}{2}+\nu^{-\frac{3}{4}}h^{-\frac{5}{4}}+\nu^{-3}h^{-5} \right).
    \end{align*}

    \smallskip

    \noindent Set 
    \[
    D_\nu (s,x):=U_0\left(\frac{s}{2\nu}\right) \test(x)-\test(x)=\frac{1}{\sqrt{2\pi}}\int e^{ix\xi}\left(e^{-i\frac{s}{2\nu}\xi^2}-1\right)\hat \phi(\xi) d\xi.
    \]
    Splitting the frequency integral at $|\xi|=\left(\frac{s}{2\nu}\right)^{-\frac{1}{2}}$, we obtain
    \begin{equation}
        \left|D_\nu(s,x)\right|\lesssim \left(\frac{s}{2\nu}\right)^\frac{1}{4}\big\|\test\big\|_{H^1}\lesssim _\eta  \left(\frac{s}{2\nu}\right)^\frac{1}{4}h^{-\frac{3}{4}}.
    \end{equation}
    Let
    \[
    I:=\int_0^\infty \int |a-b|(x+s)\big|D_\nu(s,x)\big|^4dxds.
    \]
    Splitting at $s=T>1$, the preceding pointwise estimate gives
    \begin{align*}
        \int_0^T \int |a-b|(x+s)\big|D_\nu(s,x)\big|^4dxds\lesssim_\eta C_{a,b}\,\frac{T^2}{\nu h^3}.
    \end{align*}
    For $s\geq T$, we use $|D_\nu(s,x)|\leq 2\mathcal{M_{\test}}(x)$ and split the $x-$integral at $-\frac{T}{2}$. For $x>-\frac{T}{2}$,
    \[
    \int_{x+T}^\infty |a-b|(y)dy\lesssim T^{-\frac{1}{2}}\big \|x(a-b)\big\|_{L^2}.
    \]
    On the complementary region, Lemma \ref{lem:free_packet_localization} and scaling gives
    \[
    \int_{x<-\frac{T}{2}} \big |\mathcal{M}_{\test} (x)\big|^4 dx\lesssim_\eta h^{-1} \int _{x<-\frac{T}{2}}\left \la \frac{x-x_0}{h}\right\ra^{-2}dx\lesssim_\eta \frac{h}{T}.
    \]
    Consequently, taking $T=\left(\nu h^3\right)^\frac{2}{5}\geq 1$ we bound
    \[
    I\lesssim_\eta C_{a,b} \left(\frac{T^2}{\nu h^3}+T^{-\frac{1}{2}}+\frac{h}{T} \right)\lesssim _\eta C_{a,b} \left(\nu h^3\right)^{-\frac{1}{5}}. 
    \]
    H\"older's inequality with respect to the measure $|a-b|(x+s)$ and $\|\test\|_{L^4}=\|\eta\|_{L^4}$ then gives
    \begin{align*}
        \int_0^\infty \int |a-b|(x+s) \big|\test(x)\big|^3 \big |D_\nu(s,x)\big|dxds\lesssim_\eta C_{a,b} \left(\nu h^3 \right)^{-\frac{1}{20}}.
    \end{align*}
    Thus
    \begin{align*}
    \bigg|\int_0^\infty \int (a-b)(x+s)\big|U_0\left(\frac{s}{2\nu}\right)\test(x)\big|^4 dxds-\int_0^\infty |\phi_{x_0,h}(x)|^4 &\big( \mathcal{A}_+(x)- \mathcal{B}_+(x)\big)dx\bigg|\lesssim_\eta C_{a,b} \left(\nu h^3\right)^{-\frac{1}{20}}.
    \end{align*}
    Finally, we bound
    \begin{align*}
     \left|\big(\mathcal{A}_+-\mathcal{B}_+\big)(x_0)-\int_0^\infty \big|\test\big|^4 \left(\mathcal{A}_+(x)-\mathcal{B}_+(x)\right)dx \right|&\leq h\|a-b\|_{L^\infty} \|x^{\frac{1}{4}}\eta\|_{L^4}^4\lesssim_\eta C_{a,b} h. 
\end{align*}
Combining the above estimates, we prove \eqref{eq:difference_est_primitive}.
\end{proof}

\medskip

Note that Proposition \ref{prop:primitive_difference} is stated with modified scattering maps with common delta strength. In order to compare modified scattering maps with different delta strengths, we use triangle inequality as well as establishing the dependence of $S_{\rho,a}$ on $\rho$. 
\begin{proposition}\label{prop:parameter_scattering_remainder}
With $\varepsilon\|\psi\|_X$ sufficiently small depending on $I,b,\beta$, and $\gamma$
\begin{equation}\label{eq:parameter_scattering_remainder}
    \sup_{\rho\in I}
    \|\partial_\rho\mathcal R_{\rho,b}(\varepsilon,\psi)
       \|_{L^\infty\times L^\infty}
    \lesssim_{I,b}\varepsilon^3\|\psi\|_X^3,
\end{equation}
where $\mathcal R_{\rho,b}$ is defined in \eqref{eq:L_qa}. 
Consequently, for $q,\widetilde q\in I$,
\begin{equation}\label{eq:remainder_difference_q}
    \|\mathcal R_{q,b}(\varepsilon,\psi)
       -\mathcal R_{\widetilde q,b}(\varepsilon,\psi)
       \|_{L^\infty\times L^\infty}
    \lesssim_{I,b}|q-\widetilde q|\varepsilon^3\|\psi\|_X^3.
\end{equation}
\end{proposition}
\begin{proof}
Using $\dft_\rho U_\rho(-t)=e^{it\xi^2}\dft_\rho$, write
\begin{align}
\mathcal R_{\rho,b}(\varepsilon,\psi)
=&-i\lambda B_\rho^*\int_0^1
 \overrightarrow{e^{it\xi^2}\dft_\rho
  [(1+b)|u_\rho|^2u_\rho]}
 \,dt \notag\\
&-i\lambda\int_1^\infty
 \mathcal D_{\rho,b}(t)B_\rho^*\vec{\mathcal N}_{\rho,b}(t)\,dt
 -i\int_1^\infty
 \mathcal D_{\rho,b}(t)B_\rho^*\mathcal E_\rho(t,\vec w_\rho)\,dt.
\label{eq:scattering_remainder_decomposition}
\end{align}
The derivative of the first line is bounded by
\eqref{eq:partial_small_time_remainder}.

We next record the bounds needed for the large-time terms.  Since
\[
    \partial_\rho\vec f_\rho =(\partial_\rho B_\rho^*)\vec w_\rho+B_\rho^*\vec z_\rho,
\]
the definition
$\Theta_{\rho,j}(t)=\lambda\int_1^t(2s)^{-1}|f_{\rho,j}(s)|^2\,ds$
gives
\begin{equation}\label{eq:partial_phase_rho}
    \partial_\rho\Theta_{\rho,j}(t)  =\lambda\int_1^t s^{-1}
     \Re \big( \overline{f_{\rho,j}(s)}\,\partial_\rho f_{\rho,j}(s)\big)\,ds.
\end{equation}
Proposition \ref{prop:quant_z_big} therefore implies
\begin{equation}\label{eq:partial_D_rho}
    \|\partial_\rho\mathcal D_{\rho,b}(t)\|_{L^\infty}
    \lesssim_{I,b}\alpha^2t^\gamma.
\end{equation}
Together with \eqref{eq:partial_N_Linfty} and
$\|\mathcal N_{\rho,b}(t)\|_{L^\infty}\lesssim_b\alpha^3t^{-\frac{3}{2}}$,
this gives
\begin{equation}\label{eq:parameter_localized_integrable}
    \int_1^\infty \|\partial_\rho(
      \mathcal D_{\rho,b}B_\rho^*\vec{\mathcal N}_{\rho,b})(t)
      \|_{L^\infty}\,dt\lesssim_{I,b}\alpha^3.
\end{equation}
It remains to treat the homogeneous remainder.  The direct estimates give
\begin{equation}\label{eq:E_rho_bound}
    \|\mathcal E_\rho(t,\vec w_\rho)\|_{L^\infty}
    \lesssim_{I,b}\alpha^3t^{-\frac{5}{4}+\beta}.
\end{equation}
We claim that
\begin{equation}\label{eq:partial_E_rho_bound}
    \|\partial_\rho\mathcal E_\rho(t,\vec w_\rho)\|_{L^\infty}
    \lesssim_{I,b}\alpha^3t^{-\frac{5}{4}+\beta+\gamma}.
\end{equation}
Since $w_\rho(t,0)=z_\rho(t,0)=0$, Propositions
\ref{prop:asy_V} and \ref{prop:partial_V_est} imply
\begin{align}
    \left\|V_\rho(t)w_\rho -\vec S_\rho\cdot\vec w_\rho
    \right\|_{L^\infty}&\lesssim_I  t^{-\frac{1}{4}}\|w_\rho\|_{H^1},
    \label{eq:V_forward_remainder}\\
    \left\| \partial_\rho\left( V_\rho(t)w_\rho -\vec S_\rho\cdot\vec w_\rho \right) \right\|_{L^\infty} &\lesssim_I t^{-\frac{1}{4}}\left(\|w_\rho\|_{H^1}+\|z_\rho\|_{H^1}
    \right). \label{eq:partial_V_forward_remainder}
\end{align}
Moreover,
\begin{align}
    \|V_\rho(t)w_\rho\|_{L^\infty} +\|\vec S_\rho\cdot\vec w_\rho\|_{L^\infty}  &\lesssim_{I,b}\alpha,
    \label{eq:V_forward_uniform}\\
    \left\| \partial_\rho \left(\vec S_\rho\cdot\vec w_\rho \right) \right\|_{L^\infty}
    &\lesssim_I \|w_\rho\|_{L^\infty}+\|z_\rho\|_{L^\infty}
    \lesssim_{I,b}\alpha t^\gamma. \label{eq:partial_Sw_uniform}
\end{align}
The elementary cubic estimate therefore yields
\begin{align*}
    &\left|t^{-1} \partial_\rho \left\{ |V_\rho w_\rho|^2V_\rho w_\rho
    -|\vec S_\rho\cdot\vec w_\rho|^2\vec S_\rho\cdot\vec w_\rho
    \right\} \right|\\
    &\qquad \lesssim t^{-1} \left(|V_\rho w_\rho|+|\vec S_\rho\cdot\vec w_\rho|\right)\left| V_\rho w_\rho-\vec S_\rho\cdot\vec w_\rho\right|\left| \partial_\rho(\vec S_\rho\cdot\vec w_\rho) \right| +|V_\rho w_\rho|^2
    \left| \partial_\rho\left(V_\rho w_\rho-\vec S_\rho\cdot\vec w_\rho \right)\right|\\
    &\qquad \lesssim_{I,b}
    \alpha^3t^{-\frac{5}{4}+\beta+\gamma}.
\end{align*}
It remains to estimate the contribution from the approximation of
$V_\rho(t)^{-1}$. By Proposition \ref{prop:partial_V_est}, this contribution, combined with the $t^{-1}$ in the term, is bounded by
\begin{align*}
    t^{-\frac{5}{4}}\left\||V_\rho(t)w_\rho|^2V_\rho(t)w_\rho
    \right\|_{H^1}+t^{-\frac{5}{4}} \left\|\partial_\rho
    \left(|V_\rho(t)w_\rho|^2V_\rho(t)w_\rho\right)
    \right\|_{H^1}.
\end{align*}
Proposition \ref{prop:partial_V_est} and the estimates for $w_\rho$ and
$z_\rho$ give
\begin{align*}
    \|V_\rho(t)w_\rho\|_{H^1}
    &\lesssim_I\|w_\rho\|_{H^1}
    \lesssim_{I,b}\alpha t^\beta,\\
    \|\partial_\rho(V_\rho(t)w_\rho)\|_{H^1}
    &\lesssim_I
    \|w_\rho\|_{H^1}
    +\|z_\rho\|_{H^1}
    \lesssim_{I,b}\alpha t^{\beta+\gamma},\\
    \|\partial_\rho(V_\rho(t)w_\rho)\|_{L^\infty}
    &\lesssim_{I,b}\alpha t^\gamma.
\end{align*}
Consequently, these estimates with Propositions \ref{prop:asy_V}, \ref{prop:homo_sobolev_est} and \ref{prop:partial_V_est}, combined with \eqref{eq:partial_V_forward_remainder} proves \eqref{eq:partial_E_rho_bound}. Combining \eqref{eq:partial_D_rho}, \eqref{eq:E_rho_bound}, and
\eqref{eq:partial_E_rho_bound}, and using
$\beta+\gamma<\frac{1}{4}$, we obtain
\[
    \int_1^\infty  \|\partial_\rho(\mathcal D_{\rho,b}B_\rho^*\mathcal E_\rho)(t) \|_{L^\infty}\,dt
    \lesssim_{I,b}\alpha^3.
\]
Together with \eqref{eq:parameter_localized_integrable} and the small-time
estimate, this proves \eqref{eq:parameter_scattering_remainder}.  The
fundamental theorem of calculus in $\rho$ gives
\eqref{eq:remainder_difference_q}.
\end{proof}

\medskip

The testing vector used in the recovery of the coefficient is
\[
    \vec h_{q,\psi}:=L_q\psi.
\]
Unitarity and \eqref{eq:dft_est_3} imply
\begin{equation}\label{eq:hpsi_L2_L1}
    \|\vec h_{q,\psi}\|_{L^2\times L^2}\lesssim\|\psi\|_{L^2},
    \qquad
    \|\vec h_{q,\psi}\|_{L^1\times L^1}\lesssim_q\|\psi\|_{H^1}.
\end{equation}
Indeed,
$\|\dft_q\psi\|_{L^1}\lesssim
\|\dft_q\psi\|_{L^2}+\|\xi\dft_q\psi\|_{L^2}
\lesssim_q\|\psi\|_{H^1}$.

\medskip

We record a direct corollary from \eqref{eq:remainder_difference_q}:

\begin{corollary}\label{cor:difference_q}
For $q,\widetilde q\in I$ and $\varepsilon\|\psi\|_X$ sufficiently small,
\begin{equation}\label{eq:q_map_pairing_general}
\begin{split}
    \left|\left\langle
       S_{q,b}(\varepsilon\psi)-S_{\widetilde q,b}(\varepsilon\psi),
       \vec h_{q,\psi}\right\rangle\right|\lesssim_{I,b}|q-\widetilde q|\bigl(\varepsilon\|\psi\|_{L^2}^2
    +\varepsilon^3\|\psi\|_X^4\bigr).
\end{split}
\end{equation}
Consequently,
\begin{equation}\label{eq:q_map_pairing_packet}
\begin{split}
    \left|\left\langle
       S_{q,b}(\varepsilon\psi_{\nu,x_0,h})
       -S_{\widetilde q,b}(\varepsilon\psi_{\nu,x_0,h}),
       \vec h_{q,\psi_{\nu,x_0,h}} \right\rangle\right|\lesssim_{q,\widetilde q,a,b,\eta}
      \|S_{q,a}-S_{\widetilde q,b}\|\bigl(\varepsilon h^{\frac{1}{2}} +\varepsilon^3\nu^4h\bigr).
\end{split}
\end{equation}
\end{corollary}
\begin{proof}
The half-line formulas in the proof of Proposition \ref{prop:recover_q} and
\[
    T_q(\xi)-T_{\widetilde q}(\xi)
    =R_q(\xi)-R_{\widetilde q}(\xi)
    =\frac{2i\xi(q-\widetilde q)}
       {(2i\xi-q)(2i\xi-\widetilde q)}
\]
give
\begin{equation}\label{eq:Lq_operator_lipschitz}
    \|(L_q-L_{\widetilde q})\psi\|_{L^2\times L^2}
    \lesssim_{I}|q-\widetilde q|\|\psi\|_{L^2}.
\end{equation}
Decompose the two maps, apply
\eqref{eq:Lq_operator_lipschitz}, \eqref{eq:remainder_difference_q}, and
\eqref{eq:hpsi_L2_L1}, and obtain \eqref{eq:q_map_pairing_general}.  Proposition
\ref{prop:stability_q} and the packet bounds \eqref{eq:scale_X} then imply \eqref{eq:q_map_pairing_packet}.
\end{proof}

\medskip

By definition \eqref{def:SaSb}, 
    \begin{align*}
        \left|\left \la \,\mapa(\varepsilon\psi)-\mapb(\varepsilon\psi) ,\,\vec{h}_\psi\right\ra\right|\lesssim \big\|\mapa-\mapb\big\|\varepsilon\|\psi\|_\Sigma \|\psi\|_{L^2}.
    \end{align*}
Combining Proposition \ref{prop:primitive_difference} and Corollary \ref{cor:difference_q}, we bound
\begin{align*}
    &\left|\left\la\mapa(\varepsilon\testhi)-\mapc(\varepsilon\testhi),\vec h_{\testhi}\right\ra\right|\\
    &\qquad \qquad  \lesssim_\eta \big\|\mapa-\mapb\big\|\varepsilon \nu h^\frac{1}{2}\,\, +\,\, C_{q,\tilde q,a,b} \big\|\mapa-\mapb\big\|\varepsilon h^\frac{1}{2}\,\,+\,\,C_{q,\tilde q, b}\big\|\mapa-\mapb\big\|\varepsilon^3 \nu^4 h.
\end{align*}

\medskip

We now prove Theorem \ref{thm:stability}. 
\begin{proof}[Proof of Theorem \ref{thm:stability}]
    We first assume that $\|\mapa-\mapb\|\leq \kappa_{q,\tilde q, a,b}$, where $\kappa_{q,\tilde q, a,b}>0$ will be chosen sufficiently small. Set
    \begin{align*}
    h=\|\mapa-\mapb\|^\frac{1}{116},\qquad \nu=\|\mapa-\mapb\|^\frac{-13}{116},\qquad \varepsilon=\|\mapa-\mapb\|^\frac{45}{116}.
\end{align*}
    Then,
    \[
    \nu h^3=\|\mapa-\mapb\|^\frac{-10}{116} \to \infty, \quad \nu h=\|\mapa-\mapb\|^\frac{-12}{116} \to \infty,\quad \text{as }\|\mapa-\mapb\|\to 0.
    \]
    Moreover, 
    \[
    \varepsilon \nu h^\frac{1}{4}=\|\mapa-\mapb\|^\frac{129}{464}\to 0,\quad \text{as }\|\mapa-\mapb\|\to 0.
    \]
    Thus, after decreasing $\kappa_{a,b}$, if necessary, all the hypotheses of Proposition \ref{prop:primitive_difference} are satisfied. The three principal errors balance:
    \[
    \|\mapa-\mapb\|\varepsilon^{-2}\nu^2h^{\frac{1}{2}}=\varepsilon^2 \nu^7 h^{\frac{3}{2}}= \left(\nu h^3\right)^{-\frac{1}{20}} =\|\mapa-\mapb\|^\frac{1}{232}. 
    \]
    The remaining terms satisfy
    \[
    \nu^{-\frac{3}{4}}h^{-\frac{5}{4}}=\|\mapa-\mapb\|^\frac{17}{232},\quad \nu^{-3}h^{-5}=\|\mapa-\mapb\|^\frac{68}{232}, \quad h=\|\mapa-\mapb\|^\frac{1}{116},
    \]
    and the two additional terms from changing $q$ to $\tilde q$ contributes
    \[
     \big\|\mapa-\mapb\big\|\varepsilon^{-2}\nu h^\frac{1}{2}=\big\|\mapa-\mapb\big\|^\frac{27}{232},\quad \big\|\mapa-\mapb\big\| \nu^5 h=\big\|\mapa-\mapb\big\|^\frac{13}{29}.
    \]
    Therefore,
    \begin{align*}
     \sup_{3h\leq x_0\leq h^{-1}}\left|\mathcal{A}_+(x_0)-\mathcal{B}_+(x_0)\right|\lesssim_\eta C_{a,b} \|\mapa-\mapb\|^\frac{1}{232}.
    \end{align*}
    When $\|\mapa-\mapb\|=0$, we carry out the preceding argument with an arbitrary parameter $\rho>0$ in place of $\|\mapa-\mapb\|$ and let $\rho\to 0$.
    
    For $x_0>h^{-1}$, Cauchy–Schwarz gives
    \begin{align*}
        \left|\mathcal{A}_+(x_0)-\mathcal{B}_+(x_0)\right|\leq \int_{x_0}^\infty |a-b|(y)dy\lesssim C_{a,b}h^{\frac{1}{2}}\lesssim C_{a,b}\|\mapa-\mapb\|^\frac{1}{232}.
    \end{align*}
    For $0<x_0\leq 3h$, we bound by
    \begin{align*}
    \left|\mathcal{A}_+(x_0)-\mathcal{B}_+(x_0)\right|\leq \left|\mathcal{A}_+(3h))-\mathcal{B}_+(3h)\right|+ 3h\|a-b\|_{L^\infty}\lesssim C_{a,b}\|\mapa-\mapb\|^\frac{1}{232}.
    \end{align*}
    Hence, we proved
    \begin{equation}
       \big\|\mathcal{A}_+-\mathcal{B}_+\|_{L^\infty(\R_+)}\lesssim C_{a,b}\|\mapa-\mapb\|^\frac{1}{232}.
    \end{equation}
    Testing instead with left-moving packets supported in $(-\infty,0)$ gives
    \begin{equation}
       \big\|\mathcal{A}_--\mathcal{B}_-\|_{L^\infty(\R_-)}\lesssim C_{a,b}\|\mapa-\mapb\|^\frac{1}{232}.
    \end{equation}
    This proves \eqref{eq:stability_A_B} with $\theta_0=\frac{1}{232}$. 

    \medskip

    Finally, suppose that $a',b'\in L^2(\R)$. The one-dimensional Gagliardo–Nirenberg interpolation inequality gives
    \begin{equation}
        \|a-b\|_{L^\infty(\R_+)}\lesssim \|a'-b'\|_{L^2(\R_+)}^\frac{2}{3}\big\|\mathcal{A}_+-\mathcal{B}_+\big\|_{L^\infty(\R_+)}^\frac{1}{3}.
    \end{equation}
    The corresponding estimate on $\R_-$ yields
    \begin{equation}
        \|a-b\|_{L^\infty(\R)}\lesssim C_{a,b} \|\mapa-\mapb\|^\frac{1}{696},
    \end{equation}
    where $C_{a,b}$ also depends on $\|a'\|_{L^2}$ and $\|b'\|_{L^2}$. 
\end{proof}


\end{document}